\documentclass[12pt]{scrartcl}

\usepackage[]{amsmath, amssymb,amsfonts,amsthm,mathtools,braket,color,enumerate}

\usepackage{subcaption}
\usepackage[margin=1in]{geometry}
\usepackage{hhline}
\usepackage{url}
\usepackage{here}
\usepackage{tikz-cd}
\usepackage{tikz}
\usetikzlibrary{arrows}
\usetikzlibrary{shapes}
\usepackage{tkz-graph}
\usepackage{hyperref}
\hypersetup{
  colorlinks   = true, %Colours links instead of ugly boxes
  urlcolor     = blue, %Colour for external hyperlinks
  linkcolor    = blue, %Colour of internal links
  citecolor   = red %Colour of citations
}
\usetikzlibrary{decorations}
\usetikzlibrary{arrows}
\tikzstyle{v} = [circle, draw, inner sep=2pt, minimum size=3pt, fill=black]
\tikzstyle{l} = [rectangle, draw, rounded corners]

\usetikzlibrary{decorations.pathreplacing}
\usepackage{ytableau}

\theoremstyle{plain}% default
\newtheorem{theorem}{Theorem}[section]
\newtheorem{lemma}[theorem]{Lemma}
\newtheorem{proposition}[theorem]{Proposition}
\newtheorem{corollary}[theorem]{Corollary}

\theoremstyle{definition}
\newtheorem{definition}[theorem]{Definition}

\newtheorem{example}[theorem]{Example}

\newtheorem{remark}[theorem]{Remark}
\newtheorem{claim}[theorem]{Claim}

\newcommand{\Z}{\mathbb{Z}}
\newcommand{\R}{\mathbb{R}}
\newcommand{\K}{\mathbb{K}}

\DeclareMathOperator{\rank}{rank}
\DeclareMathOperator{\Der}{Der}
\DeclareMathOperator{\cl}{cl}
\newcommand{\A}{\mathcal{A}}
\newcommand{\F}{\mathcal{F}}
\newcommand{\mcH}{\mathcal{H}}
\newcommand{\mcK}{\mathcal{K}}
\newcommand{\mcP}{\mathcal{P}}

\begin{document}

\title{MAT-free graphic arrangements and a characterization of strongly chordal graphs by edge-labeling}
\author{
Tan Nhat Tran
\thanks{
Fakult\"at f\"ur Mathematik, Ruhr-Universit\"at Bochum, D-44780 Bochum, Germany.
E-mail address: tan.tran@ruhr-uni-bochum.de. 
}\and
Shuhei Tsujie
\thanks{Department of Mathematics, Hokkaido University of Education, Asahikawa, Hokkaido 070-8621, Japan. 
E-mail address: tsujie.shuhei@a.hokkyodai.ac.jp}
}

\date{}

\maketitle

\begin{abstract}
Ideal subarrangements of a Weyl arrangement are proved to be free by the multiple addition theorem (MAT) due to Abe-Barakat-Cuntz-Hoge-Terao (2016). 
They form a significant class among Weyl subarrangements that are known to be free so far.
The concept of MAT-free arrangements was introduced recently by Cuntz-M{\"u}cksch (2020) to capture a core of the MAT, which enlarges the ideal subarrangements from the perspective of freeness. 
The aim of this paper is to give a precise characterization of the MAT-freeness in the case of type $A$ Weyl subarrangements (or graphic arrangements).
It is known that  the ideal and free graphic arrangements correspond to the unit interval and chordal graphs respectively. 
We prove that a graphic arrangement is MAT-free if and only if the underlying graph is strongly chordal. 
In particular, it affirmatively answers a question of Cuntz-M{\"u}cksch that MAT-freeness is closed under taking localization in the case of graphic arrangements.
\end{abstract}

{\footnotesize \textit{Keywords}: 
Hyperplane arrangement, free arrangement, MAT-free arrangement, ideal subarrangement, graphic arrangement, strongly chordal graph, edge-labeling of graph
}

{\footnotesize \textit{2020 MSC}: 
52C35, %Arrangements of points, flats, hyperplanes [52Cxx Discrete geometry]
%32S22,  %Relations with arrangements of hyperplanes [32Sxx Singularities]
%05C15, %Coloring of graphs and hypergraphs [05Cxx Graph theory]
%05C22, %Signed and weighted graphs [05Cxx Graph theory]
%20F55,  %Reflection and Weyl groups [20Fxx Special aspects of infinite or finite groups]
13N15, %Derivations [13Nxx Differential algebra]
05C78 %Graph labelling
}

\tableofcontents

\section{Introduction}
\label{sec:intro}
\quad
A hyperplane arrangement is said to be  \emph{free} if its logarithmic derivation module is a free module \cite{T80, OT92} (cf. Definition \ref{def:free-arr}). 
An important class of  free arrangements is that of Weyl arrangements defined by the positive roots of irreducible root systems. 
There has been considerable interest  in finding and  characterizing  free subarrangements of a Weyl arrangement by combinatorial structures. 
In the case of type $A$, Weyl subarrangements are completely determined by \emph{graphic arrangements} whose freeness is fully characterized by \emph{chordal graphs} \cite{St72, ER94} (cf. Theorem \ref{thm:free-chordal}). 
Several certain cases of  type $B$  were studied in the connection with signed graphs \cite{ER94, STT19, TT20} yet no complete characterization is known. 
Arguably, the most fundamental and significant class of Weyl subarrangements known to be free is that of \emph{ideal subarrangements} derived by the  multiple addition theorem (MAT) due to Abe-Barakat-Cuntz-Hoge-Terao \cite{ABCHT16} (cf. Theorem \ref{thm:ideal-free}). 
In fact, the MAT applies in a more general setting; based on that Cuntz-M{\"u}cksch  \cite{CM20} introduced a new class of free arrangements, the so called \emph{MAT-free arrangements}  (cf. Definition \ref{def:MAT-partition-free}) which generalizes the ideal subarrangements. 
 
We are especially interested in characterizations of freeness-related properties of type $A$ Weyl subarrangements ($=$ graphic arrangements) in which considerable power and development of graph theory would be brought to bear.
It is known that the ideal graphic arrangements are parametrized by \emph{unit interval graphs} e.g., \cite{TT21}
(cf. Theorem  \ref{thm:ideal-UI}). 
Our main result is a complete characterization of MAT-free graphics arrangements by means of  \emph{strongly chordal graphs} (cf. Definition \ref{def:strongly-chordal}). 
We summarize the results in Table \ref{tab:concepts}. 
We find it interesting that the concept of MAT-freeness which was recently introduced is captured by the notion of strongly chordal graphs which appeared  much earlier in literature.
This can also be regarded as an analogue of the  classical theory of freeness and chordality. 
We will see in \S \ref{subsec:MAT-S-PEO} that many important concepts in the classical theory such as simplicial vertices and perfect elimination orderings of chordal graphs have their analogous MAT- versions. 
\begin{table}[htbp]
\centering
{\renewcommand\arraystretch{1.5} 
\begin{tabular}{c|c|c}
Graph class & Weyl subarrangement class & Location \\
\hline\hline
chordal  &  free   & \cite{St72}, {\cite[Theorem 3.3]{ER94}}\\
\hline
\textbf{strongly chordal} &   \textbf{MAT-free}  & \textbf{Theorem \ref{thm:MAT-free-strong-chordal}}\\
\hline
unit interval &  ideal & e.g., {\cite[Theorem 16]{TT21}}\\
\end{tabular}
}
\bigskip
\caption{Interplay between graphs and Weyl arrangements in type $A$. 
The third row shown in bold indicates the main result of the paper.}
\label{tab:concepts}
\end{table}

Our main result has a number of applications.
From the viewpoint of arrangement theory, it gives an affirmative answer to a question of Cuntz-M{\"u}cksch in the case of graphic arrangements  that MAT-freeness is closed under taking localization  (cf. Corollary \ref{cor:close-localization}). 
Thanks to the relation $\{\mbox{unit interval graphs}\} \subsetneq \{\mbox{strongly chordal graphs}\}$, it also gives a different and graphical proof that the ideal graphic arrangements are MAT-free (cf. \S\ref{sec:rem}(A)). 
From the viewpoint of graph theory, our main result contributes a new characterization of strongly chordal graphs via a special type of edge-labelings, which we shall call \emph{MAT-labelings} (cf. Definition \ref{definition MAT-labeling}). 

A key ingredient in our proof that strong chordality implies MAT-freeness (the harder part) is a characterization of strongly chordal graphs by their \emph{clique intersection posets} due to Nevries-Rosenke \cite{nevries2015characterizing-dam} (cf. Theorem \ref{Nevries-Rosenke}). 
Our strategy is to construct an MAT-labeling for a given strong chordal graph with building blocks complete induced subgraphs of the graph.
The clique intersection poset of a chordal graph consists of all intersections of maximal cliques of the graph which serves as essential machinery in the construction.
We believe that it is worth pursuing further the notion of clique intersection poset for MAT-freeness of larger class of arrangements (see \S\ref{sec:rem}(F) for more details).

The remainder of the paper is organized as follows. 
In \S \ref{subsec:free-arr}, we recall the definitions and basic facts of free arrangements and chordal graphs. 
In \S \ref{subsec:MATfree-arr}, we recall the definitions of MAT-free arrangements and strongly chordal graphs, and give the statement of our main result. 
In \S \ref{sec:more}, we recall some other useful characterizations of (strongly) chordal graphs.
In \S \ref{subsec:MAT-labeling of graphs}, we introduce the notion of MAT-labeling of graphs. 
In \S \ref{subsec:MAT-free implies strongly chordal}, we prove the ``only if" part of our main Theorem \ref{thm:MAT-free-strong-chordal} (MAT-freeness implies strong chordality).
In \S \ref{subsec:MAT-S-PEO}, we introduce the notions of MAT-simplicial vertices and MAT-perfect elimination orderings.
In \S \ref{sub:if-part}, we prove the ``if" part of our main Theorem \ref{thm:MAT-free-strong-chordal} (strong chordality implies MAT-freeness). 
The proof the main theorem is also included with an example in the end of this section. 
Finally, in \S\ref{sec:rem}, we address some further remarks and suggest some problems for future research.
%******************************************************************************** 

\section{Arrangements, graphs and statement of the main result}
\label{sec:main-result}
\subsection{Free arrangements}
\label{subsec:free-arr}
\quad
We first review some basic concepts and preliminary results on free arrangements. Our standard reference is \cite{OT92}.
Let $\Bbb K$ be a field, $\ell$ be a positive integer and $V = \Bbb K^\ell$ be the  $ \ell $-dimensional vector space over $ \mathbb{K} $. 
A \textbf{hyperplane} in $V$ is a  \textit{linear} subspace of codimension one of $V$.
An \textbf{arrangement} is a finite set of hyperplanes in $V$. 
Let $\A$ be an arrangement in $V$. 
Define the \textbf{intersection lattice} $ L(\mathcal{A}) $ (of flats) of $ \mathcal{A} $ by 
\begin{align*}
L(\mathcal{A}) \coloneqq \Set{\bigcap_{H \in \mathcal{B}}H  |  \mathcal{B} \subseteq \mathcal{A} },
\end{align*}
where the partial order is given by reverse inclusion $X\le Y\Leftrightarrow Y\subseteq X$ for $X, Y \in L(\A)$. 
We agree that $V $ is the unique minimal element in $ L(\mathcal{A}) $ as the intersection over the empty set. 
Thus $ L(\mathcal{A}) $ is a geometric lattice which can be equipped with rank function $ r(X) \coloneqq \operatorname{codim}(X) $ for $X \in L(\mathcal{A})$. 
We also define $\rank(\A)$ as the rank of the maximal element $\bigcap_{H \in \mathcal{A}}H$ of $L(\A)$.

The \textbf{characteristic polynomial} $ \chi_{\mathcal{A}}(t) \in \mathbb{Z}[t] $ of $ \mathcal{A} $ is defined by
\begin{align*}
\chi_{\mathcal{A}}(t) \coloneqq \sum_{X \in L(\mathcal{A})}\mu(X)t^{\dim X}, 
\end{align*}
where $ \mu $ denotes the \textbf{M\"{o}bius function} $ \mu \colon L(\mathcal{A}) \to \mathbb{Z} $ defined recursively by 
\begin{align*}
\mu\left( V\right)= 1 
\quad \text{ and } \quad 
\mu(X) = -\sum_{\substack{Y \in L(\mathcal{A}) \\ X \subsetneq Y}}\mu(Y). 
\end{align*}

Let $ \{x_{1}, \dots, x_{\ell}\} $ be a basis for the dual space $V^{\ast} $ and let $ \textbf{S}=\mathbb{K}[x_{1}, \dots, x_{\ell}] $. 
The \textbf{defining polynomial} $Q(\A)$ of $\A$ is given by
$$Q=Q(\A)\coloneqq \prod_{H \in \A} \alpha_H \in  \textbf{S},$$
where $ \alpha_H=a_1x_1+\cdots+a_\ell x_\ell$  $(a_i\in \mathbb{K})$ satisfies $H = \ker \alpha_H$. 
The \textbf{module $D(\mathcal{A}) $ of logarithmic derivations}  is defined by 
$$D(\A)\coloneqq  \{ \theta\in \Der(\textbf{S}) \mid \theta(Q) \in Q \textbf{S}\},$$
where $ \Der( \textbf{S}) = \{ \varphi: \textbf{S}\to  \textbf{S} \mid \mbox{$\varphi$ is $ \Bbb K$-linear,
$\varphi(fg)=f\varphi(g)+g\varphi(f)$ for any $f,g \in  \textbf{S}$}\}$ is the set of all derivations of $ \textbf{S} $ over $\mathbb{K}$. 
Note that $\Der( \textbf{S})$ is a free $ \textbf{S}$-module with basis $\{\partial/\partial x_1,\ldots , \partial/\partial x_\ell\}$ consisting of the usual partial derivatives. 
A non-zero element $\varphi= f_1\cdot\partial/\partial x_1+\cdots+f_\ell\cdot\partial/\partial x_\ell\in  \Der( \textbf{S})$ is \textbf{homogeneous of degree $b$} written  $\deg\varphi=b$ if each non-zero polynomial $f_i\in  \textbf{S}$ for $1 \le i \le \ell$ is homogeneous of degree $b$. 

\begin{definition}[Free arrangements and their exponents \cite{T80, OT92}]
\label{def:free-arr}
An arrangement $ \mathcal{A} $ is called \textbf{free} with the multiset $ \exp(\mathcal{A}) = \{d_{1}, \dots, d_{\ell}\} $ of \textbf{exponents} if $D(\A)$ is a free $S$-module with a homogeneous basis $ \{\theta_{1}, \dots, \theta_{\ell}\}$ such that $ \deg \theta_{i} = d_{i} $ for each $ i $.  
\end{definition}
Remarkably, when an arrangement is free, the exponents turn out to be the roots of the characteristic polynomial due to  Terao. 
\begin{theorem}[Factorization theorem \cite{T81}, {\cite[Theorem 4.137]{OT92}}]\label{thm:Factorization}
If $\A$ is free with $\exp(\A) =\{d_{1}, \dots, d_{\ell}\} $, then 
$$\chi_\A(t)= \prod_{i=1}^\ell (t-d_i).$$
\end{theorem}

In general it is very hard to characterize the freeness of an arrangement by combinatorial data. 
It is only possible in some very special class of arrangements, for example, graphic arrangements which we shall recall shortly. 
Let $G$ be a simple graph (i.e., no loops and no multiple edges) with vertex set $V_G = \{v_{1}, \dots, v_{\ell}\}$ and edge set $E_G$.
The  \textbf{graphic arrangement} $\A_G$ in $\Bbb K^\ell$ is defined by
$$\A_G:= \{x_i - x_j=0 \mid \{v_i,v_j\} \in E_G \}.$$
A simple graph is  \textbf{chordal} if it does not  contain an induced cycle of length greater than three, or $C_{n}$-free for all $n >3$ in shorthand notation.
The freeness of graphic arrangements is completely characterized by chordality.
\begin{theorem}[Freeness and chordality \cite{St72}, {\cite[Theorem 3.3]{ER94}}]
\label{thm:free-chordal}
Let $G$ be a simple graph. 
The graphic arrangement $\A_G$ is free if and only if $G$ is chordal.
\end{theorem}

\subsection{MAT-free arrangements and the main result}
\label{subsec:MATfree-arr}
\quad
Now we recall the concept of MAT-free arrangements following \cite{CM20}.
For $X \in L(\A)$, we define the \textbf{localization} of $\A$ on $X$  by 
$${\A}_X := \{ K \in {\A} \mid X \subseteq K\},$$
and define the \textbf{restriction} ${\A}^{X}$ of ${\A}$ to $X$ by 
$${\A}^{X}:= \{ K \cap X \mid K \in{\A }\setminus {\A}_X\}.$$

For a positive integer $n$, denote $[n]:=\{1,2,\ldots,n\}$. 
\begin{definition}[MAT-partition and MAT-free arrangements {\cite[Lemma 19 and Definition 20]{CM20}}]
\label{def:MAT-partition-free}
Let $\A$ be a nonempty arrangement. 
A partition (disjoint union of nonempty subsets) $\pi= (\pi_{1}, \dots, \pi_{n}) $ of $\A$ is called an \textbf{MAT-partition} if  the following three conditions hold for every $ k \in [n] $. 
\begin{enumerate}[(MP1)]
\item\label{definition MAT partition 1} $ \rank(\pi_{k}) = |\pi_{k}| $.
\item\label{definition MAT partition 2} There does not exist $ H^{\prime} \in \mathcal{A}_{k-1} $ such that $ \bigcap_{H \in \pi_{k}}H \subseteq H^{\prime} $, where $ \mathcal{A}_{k-1} \coloneqq \pi_{1} \sqcup \dots \sqcup \pi_{k-1} $ (disjoint union) and $ \A_{0} \coloneqq \varnothing_\ell $ is the $\ell$-empty arrangement. 
\item\label{definition MAT partition 3} For each $ H \in \pi_{k} $, $ |\mathcal{A}_{k-1}| - |(\mathcal{A}_{k-1} \cup \{H\})^{H}| = k-1 $. 
\end{enumerate}
An arrangement  is called \textbf{MAT-free} if it is empty or  admits an MAT-partition. 
\end{definition}

All irreducible complex reflection arrangements that are MAT-free were characterized in  \cite{CM20}.
The name ``MAT-free arrangement" is made by inspiration of the multiple addition theorem due to Abe-Barakat-Cuntz-Hoge-Terao. 
\begin{theorem}[Multiple addition theorem (MAT) {\cite[Theorem 3.1]{ABCHT16}}]
\label{thm:MAT}
Let $\A'$ be a free arrangement with  $\exp(\A') =\{d_1, \ldots, d_\ell\}_{\le}$ (i.e., $d_{1}\le \dots\le d_{\ell}$), and $p \in [\ell]$ the number of maximal exponents.
 Let $H_1,\ldots,H_q\notin \A'$ be hyperplanes. Define
  $
\A''_j:=(\A'\cup\{H_j\})^{H_j}.
$
Assume that the following three conditions are satisfied:
\begin{enumerate}[(1)]
\item $X:=H_1 \cap \cdots \cap H_q$ is $q$-codimensional.
\item  $X \nsubseteq \bigcup_{H \in \A'}H$. 
\item $|\A'|- |\A_j''|=d$ $(j \in [q])$.
\end{enumerate}
Then $q \le p$ and $\A:=\A'\cup\{H_1,\ldots,H_q\}$ is free with $\exp(\A)=\{d_1, \ldots, d_{\ell-p}, d^{p-q}, (d+1)^q\}$.
Here $d^e$ means $d$ appears $e\ge0$ times  in the multiset of exponents.
\end{theorem}

The addition of  $\{H_1,\ldots,H_q\}$ to the arrangement $\A'$ resulting in $\A=\A'\cup\{H_1,\ldots,H_q\}$ in Theorem \ref{thm:MAT} is sometimes called an \textbf{MAT-step} \cite[Definition 2.11]{MR21}.
Thus any MAT-free arrangement is a free arrangement which can be constructed inductively from the empty arrangement by MAT-steps. 
We save the information for later use (see also \cite[Remark 2.15]{MR21}).
\begin{corollary}
\label{cor:MAT}
Suppose that $ \mathcal{A} $ is MAT-free with MAT-partition $ \pi=(\pi_{1}, \dots, \pi_{n}) $. 
Then the following statements hold. 
\begin{enumerate}[(1)]
\item For each  $ k \in [n] $, $\A_k$ is MAT-free with MAT-partition $(\pi_{1}, \dots, \pi_{k}) $.
\item $\A$ is free with $ \exp(\mathcal{A}) = \{d_{1}, \dots, d_{\ell}\}_{\leq} $ given by the block sizes of the dual partition of  $ \pi$, that is, $ d_{i} =|\{k \mid |\pi_{k}| \geq \ell-i+1\} |$ for all $i \in [n].$
\end{enumerate}
\end{corollary}

A remarkable application of the MAT is an affirmative answer for a conjecture of Sommers-Tymoczko \cite{ST06} on the freeness of ideal subarrangements of Weyl arrangements. 

Let us recall it. 
Let $\K=\R$ and $V = \Bbb R^\ell$ with the standard inner product $(\cdot,\cdot)$.
 Let $\Phi$ be an irreducible (crystallographic) root system in $V$, with a fixed positive system $\Phi^+ \subseteq \Phi$ and the associated set of simple roots $\Delta := \{\alpha_1,\ldots,\alpha_\ell \}$. 
 For  $\alpha \in  \Phi$, define $H_{\alpha} :=\{x\in V \mid(\alpha,x)=0\}.$ 
 For $\Psi\subseteq\Phi^+$, the \textbf{Weyl subarrangement} $\A_{\Psi}$ is  defined by $\A_{\Psi}:= \{H_{\alpha} \mid \alpha\in\Psi\}$. 
In particular, $\A_{\Phi^+}$ is called the \textbf{Weyl arrangement}. 

Define the partial order $\ge$ on $\Phi^+$ as follows: $\beta_1 \ge \beta_2$ if $\beta_1-\beta_2 \in\sum_{i=1}^\ell \Z_{\ge 0}\alpha_i$. 
A subset $I\subseteq\Phi^+$ is an \textbf{ideal} of $\Phi^+$  if for $\beta_1,\beta_2 \in \Phi^+$, $\beta_1 \ge \beta_2, \beta_ 1 \in I$ implies $\beta_2 \in I$. 
For an ideal $I\subseteq\Phi^+$, the corresponding Weyl subarrangement $\A_{I}$ is called the \textbf{ideal subarrangement}.

\begin{theorem}[Ideal MAT-free theorem {\cite[Theorem 1.1]{ABCHT16}}]
\label{thm:ideal-free}
Any ideal subarrangement $\A_{I}$ is MAT-free, hence free.
\end{theorem}

In this paper we are mainly interested in the graphic arrangements hence root systems of type $A$. 
We will use the following construction of type $A$ root systems.
Let $\{\epsilon_1, \ldots, \epsilon_{\ell}\}$ be an orthonormal basis for $V$, and define $U : = \{ \sum_{i=1}^{\ell} r_i\epsilon_i  \in V\mid \sum_{i=1}^{\ell} r_i=0\} \simeq \R^{\ell-1}$. 
 The set $\Phi(A_{\ell-1}) = \{\pm(\epsilon_i - \epsilon_j) \mid 1 \le i<j \le \ell\}$ 
is a root system of type $A_{\ell-1}$ in $U$, with a positive system
$
\Phi^+(A_{\ell-1}) =  \{ \epsilon_i-\epsilon_j \mid 1 \le i <  j \le \ell\}
$
 and the associated set of simple roots $\Delta(A_{\ell-1}) = \{\alpha_i:=\epsilon_i - \epsilon_{i+1} \mid 1 \le i  \le \ell-1\}$.
Thus one can see that there is a one-to-one correspondence between the graphic arrangements in $\R^\ell$ and type $A_{\ell-1}$ Weyl subarrangements.

In the case of type $A$, the ideal subarrangements can be parametrized by unit interval graphs. 
Recall that a simple graph is a \textbf{unit interval graph} if it is chordal and $\mbox{(claw, net, $3$-sun)-free}$  (see Figure \ref{fig:graph-obstacle}).

\begin{figure}[htbp]
\centering
\begin{subfigure}{.25\textwidth}
  \centering
\begin{tikzpicture}
\draw (0,0) node[v](1){};
\draw (0,1) node[v](2){};
\draw (-0.865,-0.5) node[v](3){};
\draw ( 0.865,-0.5) node[v](4){};
\draw (2)--(1)--(3);
\draw (1)--(4);
%\draw (0,-1) node(){claw};
\end{tikzpicture}
  \caption*{claw}
  \label{fig:claw}
\end{subfigure}%
\begin{subfigure}{.25\textwidth}
  \centering
\begin{tikzpicture}
\draw (0,0) node[v](1){};
\draw (1,0) node[v](2){};
\draw (0.5,0.865) node[v](3){};
\draw (-0.865,-0.5) node[v](4){};
\draw ( 1.865,-0.5) node[v](5){};
\draw (0.5,1.865) node[v](6){};
\draw (4)--(1)--(3)--(6);
\draw (1)--(2)--(3);
\draw (2)--(5);
%\draw (0.5,-1) node(){net};
\end{tikzpicture}
  \caption*{net}  
  \label{fig:net}
\end{subfigure}%
\begin{subfigure}{.25\textwidth}
  \centering
\begin{tikzpicture}
\draw (0,0) node[v](1){};
\draw (1,0) node[v](2){};
\draw (2,0) node[v](3){};
\draw (0.5,0.865) node[v](4){};
\draw (1.5,0.865) node[v](5){};
\draw (1,1.73) node[v](6){};
\draw (6)--(4)--(1)--(2)--(3)--(5)--(6);
\draw (4)--(2)--(5)--(4);
%\draw (1,-0.5) node(){$ 3 $-sun};
\end{tikzpicture}
  \caption*{$3$-sun}  
  \label{fig:3sun}
\end{subfigure}%
\begin{subfigure}{.25\textwidth}
  \centering
\begin{tikzpicture}
\draw (0,0) node[v](x1){};
\draw (1,0) node[v](x2){};
\draw (1,1) node[v](x3){};
\draw (0,1) node[v](x4){};
\draw (0.5,-0.865) node[v](y1){};
\draw (1.865,0.5) node[v](y2){};
\draw (0.5,1.865) node[v](y3){};
\draw (-0.865,0.5) node[v](y4){};
\draw (x1)--(x2)--(x3)--(x4)--(x1)--cycle;
\draw (x1)--(x3);
\draw (x2)--(x4);
\draw (x1)--(y1)--(x2)--(y2)--(x3)--(y3)--(x4)--(y4)--(x1)--cycle;
%\draw (0.5,-1.5) node(){$ 4 $-sun};
\end{tikzpicture}
  \caption*{$4$-sun}  
  \label{fig:4sun}
\end{subfigure}
\caption{Some obstructions to unit interval and strongly  chordal graphs.}
\label{fig:graph-obstacle}
\end{figure}
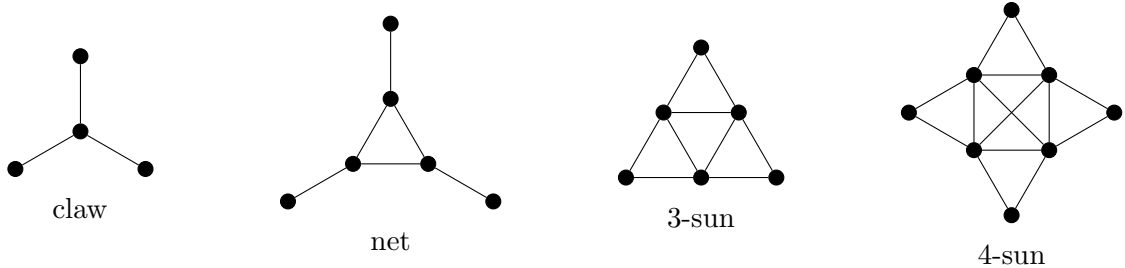

 \begin{theorem}[Ideals and unit interval graphs e.g., {\cite[Theorem 16]{TT21}}]
\label{thm:ideal-UI}
Let $G$ be a simple graph with $\ell$ vertices. 
There exists a vertex-labeling of $G$ using elements from $[\ell]$ so that the graphic arrangement $\A_G$ is an ideal subarrangement of the Weyl arrangement $\A_{\Phi^+(A_{\ell-1})}$ if and only if $G$ is a unit interval graph.
\end{theorem}

In the study of interplay between arrangements and graphs it is thus natural to ask which graph class corresponds to the MAT-free graphic arrangements? 
An answer to this question concerns strongly chordal graphs, a class squeezed between the classes of unit interval and chordal graphs.
\begin{definition}[Strongly chordal graphs e.g., \cite{Far83}]
\label{def:strongly-chordal}
An \textbf{$n$-sun (or trampoline)} $ S_{n} $ ($n \ge 3$) is a (chordal) graph with vertex set $V_{S_{n}} = \{u_{1}, \dots, u_{n}\} \cup \{v_{1}, \dots, v_{n}\} $ and edge set 
\begin{align*}
E_{S_{n}} = \Set{\{u_{i},u_{j}\} | 1 \leq i < j \leq n} \cup \Set{\{v_{i}, u_{j}\} | 1 \leq i \leq n, j \in \{i, i+1\}}, 
\end{align*}
where we consider $ u_{n+1} = u_{1} $. 
 A simple graph is  \textbf{strongly chordal} if it is chordal and $n$-sun-free for $n\ge3$. 
 See Figure \ref{fig:graph-obstacle} for the $n$-suns with $n=3,4$.
\end{definition}

 We are ready to state the main result of the paper (whose proof will be presented in the end of \S \ref{sub:if-part}).
 \begin{theorem}[MAT-freeness and strong chordality]
\label{thm:MAT-free-strong-chordal}
Let $G$ be a simple graph. 
The graphic arrangement $\A_G$ is MAT-free if and only if $G$ is strongly chordal.
\end{theorem}

Cuntz-M{\"u}cksch  \cite[Problem 48]{CM20} asked if the class of MAT-free arrangements is closed under taking localization.  
 An important consequence of our Theorem \ref{thm:MAT-free-strong-chordal} is an affirmative answer for this question in the case of graphic arrangements.
 \begin{corollary}
 \label{cor:close-localization}
MAT-freeness of graphic arrangements is closed under taking localization.  
\end{corollary}
 
%******************************************************************************** 

\section{More on (strongly) chordal graphs}
\label{sec:more}
\quad
In this section, we recall some other characterizations of (strongly) chordal graphs that will be useful for our discussion later.

First we collect some terminology and notation from graph theory. 
Let $ G = (V_{G}, E_{G}) $ be a simple graph. 
For $S \subseteq V_{G}$, denote by $G[S] = (S, E_{G[S]})$ where $E_{G[S]} =\{ \{u,v\} \in E_G \mid u, v \in S\}$ the \textbf{(vertex-)induced subgraph} of $S$.
If $v$ is a vertex of $G$ (sometimes $v \in G$ is used) then by $G\setminus  v$ we mean the induced subgraph $G[V_G\setminus\{v\}]$. 
For $ F \subseteq E_{G} $, define the subgraphs $ G_{F} \coloneqq (V_{G}, F) $ and $ G\setminus F \coloneqq (V_{G}, E_{G} \setminus F) $.
If $e$ is an edge of $G$ (sometimes $e \in G$ is used) then by $G\setminus e$ we mean the subgraph $G \setminus \{e\} $. 

An \textbf{$n$-cycle} $C_{n}$ $(n\ge3)$ is a graph with vertex set $\{v_1, v_2, \ldots, v_n\}$ and edge set $\{\{v_i, v_{i+1}\} \mid 1 \le i\le n\}$ where $v_{n+1}=v_1$. 
The $3$-cycle is also called a \textbf{triangle}.
The  \textbf{length}  of a cycle is its number of edges. 
A \textbf{chord} of $C$ is an edge not in the edge set of $C$ whose endvertices are in the vertex set.

A \textbf{clique} of $G$ is a subset of $V_{G}$ such that every two distinct vertices are adjacent. 
For each $v \in V_{G}$, its  \textbf{neighborhood} in $G$ is defined by $ N_{G}(v) =\{ u \in V_{G} \mid \{u,v\} \in E_G\}$. 
A vertex $ v \in V_{G} $ is called \textbf{simplicial} if  its neighborhood is a clique. 
An ordering $ (v_{1}, \dots, v_{\ell}) $ of $ G $ (a linear order on $ V_{G} $) is called a \textbf{perfect elimination ordering (PEO)}   if $ v_{i} $ is simplicial in the induced subgraph $ G[\{v_{1}, \dots, v_{i}\}] $ for each $ i \in [\ell] $. 
The following characterization of chordal graphs is useful to determine the exponents of the corresponding graphic arrangement (e.g., \cite[Lemma 3.4]{ER94}).

\begin{theorem}[Chordality and PEO {\cite{fulkerson1965incidence-pjom}}]
  \label{thm:chordal-PEO}
A simple graph is chordal if and only if it has a perfect elimination ordering. 
\end{theorem}

Let $ a,b \in V_{G} $ be two distinct vertices which belong to the same connected component of $ G $. 
A subset $ S \subseteq V_{G} $ is called an \textbf{$ (a,b) $-separator} if $ a $ and $ b $ belong to different connected components of $ G[V_{G}\setminus S] $. 
An $ (a,b) $-separator is \textbf{minimal} if it does not properly contain any $ (a,b) $-separator. 
A \textbf{minimal vertex separator} is a minimal $ (a,b) $-separator for some $ a,b \in V_{G} $. 
The following characterization of chordality will also be useful for some inductive arguments later. 
\begin{theorem}[{\cite[Theorem 1]{dirac1961rigid-aadmsduh}}]\label{Dirac minimal vertex separator}
A simple graph is chordal if and only if every minimal vertex separator is a clique. 
\end{theorem}

A \textbf{maximal clique} is a clique that it is not a subset of any other clique. 
A \textbf{largest (or maximum) clique} is a clique that has the largest possible number of vertices. 
Denote by $\mcK(G)$ the set of all maximal cliques of $G$.

Let $G$ be a chordal graph. 
Let $ \mathcal{P}_{G} $ be the poset consisting of  all \emph{possibly-empty} intersections of maximal cliques of $ G $, i.e, 
$$
 \mathcal{P}_{G} = \Set{\bigcap_{C \in \mathcal{B}}C  |   \varnothing \ne \mathcal{B} \subseteq \mcK(G) },
$$
where the partial order is given by inclusion $X_1\le X_2\Leftrightarrow X_1\subseteq X_2$ for $X_1, X_2 \in \mathcal{P}_{G}$. 
We call $ \mathcal{P}_{G} $ the \textbf{clique intersection poset}\footnote{The poset $ \mathcal{P}_{G}  \setminus \{\varnothing\}$ where  $\varnothing$ is the empty set was first defined in \cite{nevries2015characterizing-dam} where its Hasse diagram is called \textbf{clique arrangement}. It is not to be confused with ``hyperplane arrangement".}
 of $ G $. 
Note that $ \mathcal{P}_{G} $  is a meet-semilattice (not necessarily graded) whose maximal elements are the maximal cliques of $G$, and minimal element $ \hat{0}:=\bigcap_{C \in \mcK(G) }C \in  \mathcal{P}_{G} $ is the clique consisting of the dominating vertices\footnote{A dominating vertex is a  vertex that is adjacent to all other vertices of the graph. The presence of minimum  element $ \hat{0}$  (possibly the empty set) is helpful for us, e.g., to define the rank of nodes in Lemma \ref{node MAT-labeling}.}
of $ G $.
We call an element of $ \mathcal{P}_{G} $ a \textbf{node}.

\begin{remark}\label{Ho and Lee}
Ho-Lee \cite[Lemma 2.1]{ho1989counting-ipl} showed that a nonempty subset $ S \subseteq V_{G} $ is a minimal vertex separator if and only if $ S = C \cap C^{\prime} $ for distinct maximal cliques $ C $ and $ C^{\prime} $ forming an edge in some clique tree of $G$.
Therefore every minimal vertex separator of $ G $ belongs to $ \mathcal{P}_{G} $. 
\end{remark}

A \textbf{$ k $-crown}\footnote{The $ k $-crown here plays a role of \textbf{bad $ k $-cycle} in  \cite{nevries2015characterizing-dam}. 
More precisely, there exists an induced $k$-crown of the clique intersection poset if and only if there exists an induced bad $k$-cycle of the clique arrangement. The bottom and top elements of a $k$-crown are  the starters and terminals of the corresponding bad $k$-cycle respectively. Also, it is not to be confused with the \textbf{$k$-crown graph} which is a graph on $\{u_{1}, \dots, u_{n}\} \cup \{v_{1}, \dots, v_{n}\}$ and with an edge from $u_i$ to $v_j$ whenever $i\ne j$.}
$(k\ge1)$
 is a poset on $ \{x_{1}, \dots, x_{k}, y_{1}, \dots, y_{k}\} $ with relations $ x_{i} < y_{i} $ and $ x_{i} < y_{i+1} $ for all $1 \le  i \le k$ (counted modulo $k$) and there are no other relations. 
 See Figure \ref{fig:crown} for the $k$-crowns with $k=3,4$.
A poset $ P $ is called \textbf{$ k $-crown-free} if there exists no induced subposet\footnote{A poset $(Q,\le_Q)$ is an \textbf{induced subposet} of a poset $(P,\le_P)$ if $Q\subseteq P$ and for any $a,b \in Q$ it holds that $a \le_Q b$ if and only if $a \le_P b.$} of $ P $ isomorphic to the $ k $-crown.  
The following characterization of strongly chordal graphs will play a crucial role in the proof of the ``if" part of our main Theorem \ref{thm:MAT-free-strong-chordal} (strong chordality implies MAT-freeness \S\ref{sub:if-part}).
\begin{theorem}[Strong chordality and clique intersection poset {\cite[Theorem 1]{nevries2015characterizing-dam}}]\label{Nevries-Rosenke}
A chordal graph $ G $ is strongly chordal if and only if its clique intersection poset $ \mathcal{P}_{G} $ is $ k $-crown-free for all $ k \geq 3 $.\footnote{It is not hard to see that  $ \mathcal{P}_{G} $ is $ k $-crown-free for all $ k \geq 3 $ if and only if $ \mathcal{P}_{G}  \setminus \{\varnothing\}$  is $ k $-crown-free for all $ k \geq 3 $.}
\end{theorem}

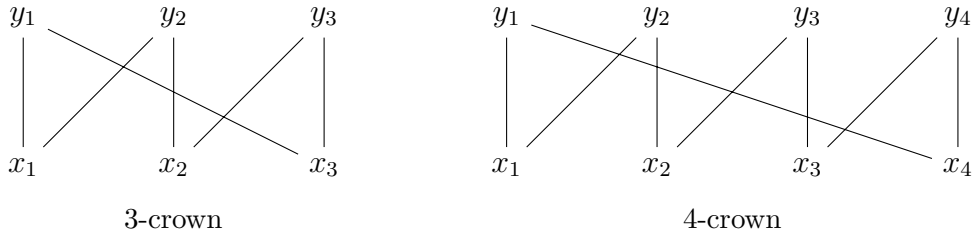
\begin{figure}[htbp]
\centering
\begin{subfigure}{.35\textwidth}
  \centering
\begin{tikzpicture}[scale=1]
  \node (x1) at (0,0) {$x_1$};
  \node (x2) at (2,0) {$x_2$};
  \node (x3) at (4,0) {$x_3$};
    \node (y1) at (0,2) {$y_1$};
  \node (y2) at (2,2) {$y_2$};
  \node (y3) at (4,2) {$y_3$};
  \draw (x1) -- (y1) -- (x3) -- (y3) -- (x2) -- (y2)-- (x1);
\end{tikzpicture}
  \caption*{$3$-crown}
  \label{fig:3cr}
\end{subfigure}%
\qquad
\begin{subfigure}{.45\textwidth}
  \centering
\begin{tikzpicture}[scale=1]
  \node (x1) at (0,0) {$x_1$};
  \node (x2) at (2,0) {$x_2$};
  \node (x3) at (4,0) {$x_3$};
    \node (x4) at (6,0) {$x_4$};
    \node (y1) at (0,2) {$y_1$};
  \node (y2) at (2,2) {$y_2$};
  \node (y3) at (4,2) {$y_3$};
    \node (y4) at (6,2) {$y_4$};
  \draw (x1) -- (y1) -- (x4) -- (y4) -- (x3) -- (y3) -- (x2) -- (y2)-- (x1);
\end{tikzpicture}
  \caption*{$4$-crown}
  \label{fig:4cr}
\end{subfigure}
\caption{Hasse diagrams of crowns.}
\label{fig:crown}
\end{figure}

%******************************************************************************** 
\section{MAT-freeness implies strong chordality}
\subsection{MAT-labeling of graphs}
\label{subsec:MAT-labeling of graphs}
\quad
In this subsection, we show that MAT-free graphic arrangements (Definition \ref{def:MAT-partition-free}) can be completely determined by a  special edge-labeling of graphs. 
First we show that the condition of being admitted a ``partition" of a (nonempty) MAT-free arrangement is actually implied by the three conditions (MP\ref{definition MAT partition 1}), (MP\ref{definition MAT partition 2}), and (MP\ref{definition MAT partition 3}). 
\begin{proposition}\label{no internal empty block}
An arrangement $ \mathcal{A} $ is MAT-free if and only if $ \mathcal{A} $ can be decomposed into a disjoint union of \emph{possibly-empty} subsets $ \pi_{1}, \dots, \pi_{n} $ of $ \mathcal{A} $ satisfying (MP\ref{definition MAT partition 1}), (MP\ref{definition MAT partition 2}), and (MP\ref{definition MAT partition 3}). 
\end{proposition}
\begin{proof}
If $ \A = \varnothing $, then the statement is clear. Suppose $ \A \ne \varnothing $. 
We only need to show the ``if" part, namely, the existence of a disjoint union of  possibly-empty subsets $ \pi_{1}, \dots, \pi_{n} $ of $ \mathcal{A} $ satisfying (MP\ref{definition MAT partition 1}), (MP\ref{definition MAT partition 2}), and (MP\ref{definition MAT partition 3}) implies the existence of an MAT-partition of $\A$.

First we show that $ \pi_{1} \neq \varnothing $. 
Suppose to the contrary that $ \pi_{1} = \varnothing $. 
Thus $n\ge2$.
Let $ k \geq 2 $ be the minimal integer such that $ \pi_{k} \neq \varnothing $. 
Then (MP\ref{definition MAT partition 3}) yields $ 0 = |\mathcal{A}_{k-1}| \geq k-1 \geq 1 $, a contradiction. 
Thus $ \pi_{1} \neq \varnothing $. 

Let $p:= \max \{ 1 \le i\le n \mid \pi_i \ne\varnothing \}$.
We will show that  $\pi_{k} \neq \varnothing $ for all $ k \in [p] $ which in turn implies that $ (\pi_{1}, \dots, \pi_{p}) $ is an MAT-partition of $ \mathcal{A} $. 
Suppose to the contrary that there exists   $2\le  k< p$ such that $ \pi_{k} = \varnothing $ and choose minimal such $k$.
Set $ \mathcal{B}^{\prime} \coloneqq \pi_{1} \cup \dots \cup \pi_{k-1} $. 
By definition, $ \mathcal{B}^{\prime} $ is MAT-free with MAT-partition $ (\pi_{1}, \dots, \pi_{k-1}) $. 
Also, the maximal exponents of $ \mathcal{B}^{\prime} $ are equal to $ k-1 $. 
Let $q:= \min \{k<i \le p\mid \pi_i \ne\varnothing \}$.
Since $ \pi_{q} \neq \varnothing $, we can take $ H \in  \pi_{q}$ and write $ \mathcal{B} \coloneqq \mathcal{B}^{\prime} \cup \{H\} $. 
It is a known fact\footnote{Here is the precise statement: ``Let  $ \mathcal{A} $ be an arrangement and let $ H \in \mathcal{A} $. If $ \mathcal{A}^{\prime} \coloneqq \mathcal{A}\setminus\{H\} $ is free with maximal exponent $ m $, then $ |\mathcal{A}^{\prime}| - |\mathcal{A}^{H}| \leq m $." 
We believe that this fact is well known among experts, but we give here a short proof for the sake of completeness. 
There exists a polynomial $ B $ such that $ \deg B = |\mathcal{A}^{\prime}| - |\mathcal{A}^{H}| $ and $ D(\mathcal{A}^{\prime})\alpha_{H} $ is contained in the ideal $(\alpha_{H}, B) \subseteq S$ \cite[Lemma 4.39 and Proposition 4.41]{OT92}. 
If $ \deg B > m $, then $ D(\mathcal{A}^{\prime})\alpha_{H} \subseteq (\alpha_{H}) $, hence $ D(\mathcal{A}^{\prime}) = D(\mathcal{A}) $. 
Therefore $ \mathcal{A} $ is free and  $\exp (\mathcal{A} )=\exp (\mathcal{A}^{\prime} )$. 
This implies $ |\mathcal{A}^{\prime}| = |\mathcal{A}| $, a contradiction. 
Thus $ |\mathcal{A}^{\prime}| - |\mathcal{A}^{H}| = \deg B \leq m $.}
 in the theory of free arrangements that $ |\mathcal{B}^{\prime}| - |\mathcal{B}^{H}| \leq k-1$.
However, the condition (MP\ref{definition MAT partition 3}) implies $ |\mathcal{B}^{\prime}| - |\mathcal{B}^{H}| = q-1 >k-1$, a contradiction. 
This completes the proof.
\end{proof}

A pair $ (G,\lambda) $ consisting of a simple graph $ G $ and a map $ \lambda \colon E_{G} \to \mathbb{Z}_{>0} $ (called (edge-)labeling) is said to be an \textbf{edge-labeled graph}. 
Now we define a labeling of graphs which characterizes the MAT-freeness of graphic arrangements.
\begin{definition}[MAT-labelings]
\label{definition MAT-labeling}
Let $ (G,\lambda) $ be an  edge-labeled graph. 
Let $ \pi_{k} \coloneqq \lambda^{-1}(k) \subseteq E_{G}$ and $ E_{k} := \pi_{1} \sqcup \dots \sqcup \pi_{k} $ for every $ k \in \mathbb{Z}_{>0} $ and $ E_{0} \coloneqq \varnothing $. 
We say that $ \lambda $ is an \textbf{MAT-labeling} if the following conditions hold for every $ k \in \mathbb{Z}_{>0} $. 
\begin{enumerate}[(ML1)]
\item\label{definition MAT-labeling forest} $ \pi_{k} $ is a forest. 
\item\label{definition MAT-labeling closure} $ \cl(\pi_{k}) \cap E_{k-1} = \varnothing $. 
Here  $\cl(F)$ for $F \subseteq E_{G}$ denotes the \textbf{closure} of $F$ in the matroid sense. Namely,
an edge $e\in E_{G}$ is in $\cl(F)$ when the two endvertices of $e$ are connected by edges in $F$.
\item\label{definition MAT-labeling triangle} Every $ e \in \pi_{k} $ forms exactly $ k-1 $ triangles ($3$-cycles) with edges in $ E_{k-1} $. 
\end{enumerate}
\end{definition}

\begin{proposition}
\label{prop:MATF=MATL}
Let $ G  $ be a simple graph.
The graphic arrangement $ \mathcal{A}_{G} $ is MAT-free if and only if $ G $ admits an MAT-labeling. 
\end{proposition}
\begin{proof}
This is a translation of Definition \ref{def:MAT-partition-free} into graphical terms with the use of Proposition \ref{no internal empty block}. 
\end{proof}

Thus characterizing the MAT-freeness of graphic arrangements amounts to characterizing the graphs having MAT-labelings. 
Here are first and simple facts on MAT-labelings.
Denote by $ K_{\ell} $ the complete graph on $ \ell $ vertices. 
\begin{proposition}\label{MAT-labeling on complete graph}
If $ \lambda $ is an MAT-labeling of $ K_{\ell} $, then $ |\pi_{k}| = \ell-k $ for all $ k \in [\ell-1] $. 
\end{proposition}
\begin{proof}
The graphic arrangement $ \mathcal{A}_{K_{\ell}} $ (in $\R^\ell$) is precisely the Weyl arrangement of type $A_{\ell-1}$ (also known as the \textbf{braid arrangement}) which has exponents $ \{0,1,2, \dots, \ell-1\}$.
Corollary \ref{cor:MAT} completes the proof.
\end{proof}

\begin{remark}
\label{rem:complete}
In fact, $ K_{\ell} $ always has an MAT-labeling $\lambda$. 
We can see this from Theorem \ref{thm:ideal-free} as $ K_{\ell} $ corresponds to a positive system (in particular, an ideal) of a root system of type $A$. 
More precisely,  $ \lambda \colon E_{K_{\ell}} \to \mathbb{Z}_{>0} $ is given by $ \lambda(\{v_{i}, v_{j}\}) = j-i$ for $1\le i < j \le \ell$ according to the height\footnote{The \textbf{height} of a positive root $\beta = \sum_{\alpha\in \Delta} c_\alpha \alpha \in \Phi^+$ is defined by $\sum_{\alpha\in \Delta} c_\alpha$.} 
 of positive roots. 
 Also, one can check directly that this labeling satisfies (ML\ref{definition MAT-labeling forest}), (ML\ref{definition MAT-labeling closure}) and (ML\ref{definition MAT-labeling triangle}).
We will see a different\footnote{Two edge-labeled graphs $(G_1,\lambda_1)$ and $(G_2, \lambda_2)$ are \textbf{isomorphic} if there exists a bijection $\sigma :V_{G_1} \to V_{G_2}$ such that $\{v_{i}, v_{j}\} \in E_{G_1}$ if and only if $\{\sigma(v_{i}), \sigma(v_{j})\} \in E_{G_2}$   with $ \lambda_1(\{v_{i}, v_{j}\}) =  \lambda_2(\{\sigma(v_{i}), \sigma(v_{j})\})$. If $G_1=G_2=G$, then we say that two labelings $\lambda_1$ and $ \lambda_2$ are the same (or isomorphic) if $(G,\lambda_1)$ and $(G, \lambda_2)$ are isomorphic.}  
(or nonisomorphic) labeling of $ K_{\ell} $ in Lemma \ref{MAT-L-PEO of complete graph}(\ref{MAT-L-PEO of complete graph 2}) (see also Figure \ref{fig:MAT-K4}). 

\end{remark}

It is important to know whether or not a restriction of an MAT-labeling is also an MAT-labeling. 
The proposition below states that it is enough to check the third condition.
\begin{proposition}\label{restriction to subset}
Let $ \lambda $ be an MAT-labeling of a simple graph $ G $ and $ F \subseteq E_{G} $. 
Then the restriction $ \lambda|_{F} $ is an MAT-labeling of the subgraph $ G_{F}= (V_{G}, F) $ if and only if $ \lambda|_{F} $ satisfies (ML\ref{definition MAT-labeling triangle}). 
\end{proposition}
\begin{proof}
Since $ (\lambda|_{F})^{-1}(k) = \pi_{k} \cap F \subseteq \pi_{k} $ and $ \pi_{k} $ is a forest, $ (\lambda|_{F})^{-1}(k) $ is also a forest. 
Moreover, $ \cl_{G_{F}}(\pi_{k} \cap F) \cap (E_{k-1} \cap F) \subseteq \cl_{G}(\pi_{k}) \cap E_{k-1} = \varnothing $. 
Thus the conditions (ML\ref{definition MAT-labeling forest}) and (ML\ref{definition MAT-labeling closure}) are automatically satisfied. 
\end{proof}

\begin{lemma}\label{restriction to the union}
Let $ \lambda $ be an MAT-labeling of a simple graph $ G $ and $ F_{1}, F_{2} \subseteq E_{G} $. 
If $ \lambda|_{F_{1}} $ and $ \lambda|_{F_{2}} $ are MAT-labelings, then $ \lambda|_{F_{1} \cup F_{2}} $ is an MAT-labeling. 
\end{lemma}
\begin{proof}
For every subset $ F \subseteq E_{G} $ and $ k \in \mathbb{Z}_{>0} $, let $ \pi_{k}^{F} \coloneqq \lambda|_{F}^{-1}(k) = \pi_{k} \cap F $.
By Proposition  \ref{restriction to subset}, it suffices to prove that $ \lambda|_{F_{1} \cup F_{2}} $ satisfies (ML\ref{definition MAT-labeling triangle}). 

Let $ e \in \pi_{k}^{F_{1}\cup F_{2}} = \pi_{k}^{F_{1}} \cup \pi_{k}^{F_{2}} $. 
Without loss of generality, we may assume $ e \in \pi_{k}^{F_{1}} $. 
Since $ \lambda|_{F_{1}} $ is an MAT-labeling, $ e $ forms at least $ k-1 $ triangles with edges in $ E_{k-1} \cap (F_{1} \cup F_{2}) $. 
Moreover, since $ \lambda $ is an MAT-labeling, $ e $ forms at most $ k-1 $ triangles with edges in $ E_{k-1} \cap (F_{1} \cup F_{2}) $. 
Therefore $ e $ forms exactly $ k-1 $ triangles with edges in $ E_{k-1} \cap (F_{1} \cup F_{2}) $. 
\end{proof}
%******************************************************************************** 
\subsection{Proof of the implication ``MAT-free $\Rightarrow$ strongly chordal"}
\label{subsec:MAT-free implies strongly chordal}
\quad
In this subsection we prove the ``only if" part of our main Theorem \ref{thm:MAT-free-strong-chordal} (MAT-freeness implies strong chordality).

First we need a few preliminary results.
Let $ G = (V_{G}, E_{G}) $ be a simple graph. 
Let $ \chi_G(t)$ be the \textbf{chromatic polynomial} of $G$ (the polynomial that counts the number of proper vertex colorings of $G$). 
It is known that $ \chi_G(t) = \chi_{ \mathcal{A}_{G} }(t) $ (e.g., \cite[Theorem 2.88]{OT92}).
Let $ \omega(G) $ denote the \textbf{clique number}, the number of vertices in a largest clique of $ G $. 
Recall that $\mcK(G)$ denotes the set of all maximal cliques of $G$.

\begin{proposition}\label{the number of maximal exponents = the number of largest cliques}
Let $ G $ be a chordal graph. Then the following statements hold. 
\begin{enumerate}[(1)]
\item The maximal exponents of $ \mathcal{A}_{G} $ are equal to $ \omega(G)-1 $. 
In addition, the number of maximal exponents of $ \mathcal{A}_{G} $ equals the number of largest cliques of $ G $.  
\item  If $ \lambda $ is an MAT-labeling of  $ G $, then the endvertices of each $ e \in  \pi_{n}$ where $n = \omega(G)-1$ are contained in a unique maximal clique of $G$. 
Furthermore, the map $ \phi : \pi_{n} \to \mcK(G)$ defined by $\phi(e)=$ the maximal clique containing the endvertices of $ e$ induces a bijection $ \pi_{n} \simeq \phi(\pi_{n})$ and $\phi(\pi_{n})=\{\mbox{all largest cliques of $ G $}\}$. 
\end{enumerate}
 
\end{proposition}
\begin{proof}
First we prove part (1).
When $ G = K_{\ell} $, $ \omega(G) = \ell $ and $ \chi_{G}(t) = t(t-1) \cdots (t-\ell+1) $. 
The assertions clearly hold. 

We may assume that $ \ell \geq 3 $ and $ G $ is not complete. 
We proceed by induction on $ \ell = |V_{G}| $. 
Since $ G $ is not complete, there exist two nonadjacent vertices $ a,b \in V_{G} $. 
Let $ S \subseteq V_{G} $ be a minimal $ (a,b) $-separator. 
Then $ V_{G} $ is decomposed as $ V_{G} = A \sqcup S \sqcup B $, where $ a \in A $ and $ b \in B $. 
Let $ G_{1} \coloneqq G[A\sqcup S] $ and $ G_{2} \coloneqq G[S \sqcup B] $. 
Note that $ G_{1} $ and $ G_{2} $ are chordal. 
Moreover, $ G[S] $ is complete (Theorem \ref{Dirac minimal vertex separator}) and $ \chi_{G}(t) = \chi_{G_{1}}(t)\chi_{G_{2}}(t)/\chi_{G[S]}(t) $ (e.g., \cite[Theorem 3]{Read68}). 

Let $ m_{1} $ and $ m_{2} $ denote the numbers of cliques consisting of $ \omega(G) $ many vertices in $ G_{1} $ and $ G_{2} $, respectively. 
Since there is no clique of $ G $ containing both vertices in $ A $ and in $ B $, the number of largest cliques of $ G $ equals $ m_{1} $ + $ m_{2} $. 
By the induction hypothesis, the chromatic polynomials of $ G_{1} $ and $ G_{2} $ can be expressed as $ \chi_{G_{1}}(t)  = (t-\omega(G)+1)^{m_{1}} f(t) $ and $ \chi_{G_{2}}(t) = (t-\omega(G)+1)^{m_{2}}g(t) $, where $ f(t), g(t) \in \mathbb{Z}[t] $ are the products of some linear factors with roots strictly smaller than $ \omega(G) -1$. 
Therefore $ \chi_{G}(t) = (t-\omega(G)+1)^{m_{1}+m_{2}}f(t)g(t)/\chi_{G[S]}(t) $. 
Since $ \chi_{G[S]}(t) = t(t-1) \cdots (t-|S|+1) $ and $ |S| < \omega(G) $ (Remark \ref{Ho and Lee}), the maximal exponents of $ G $ are equal to $ \omega(G)-1 $  and the number of maximal exponents of $ G $ is $ m_{1} + m_{2} $. 

Now we prove part (2). If $ n = 0 $, the assertions hold trivially. 
If $ n = 1 $, then $ G $ is a forest and $ \lambda $ is a constant labeling whose value is $ 1 $. 
Hence the assertions also hold. 

Now suppose $ n \geq 2 $. 
Let $ e \in \pi_{n} $. 
Clearly, there exists $ C\in \mcK(G)$ such that $e \in G[C]$.
Suppose that there exist two distinct $ C_{1} ,C_{2} \in \mcK(G)$ such that $ e\in G[C_{1} \cap C_{2}] $. 
Since $ \lambda_{E_{G \setminus e}} $ is an MAT-labeling, $ \mathcal{A}_{G\setminus e} $ is free and hence $ G\setminus e $ is chordal. 
By the maximality of $ C_{1} ,C_{2}$, there exist $u\in C_{1}\setminus C_{2} $ and $v\in C_{2}\setminus C_{1} $ such that $\{u,v\} \notin E_G$.
Then $u,v$ and the endvertices of $ e$ form a $4$-cycle  which is chordless in $ G\setminus e $, a contradiction. 
Thus the endvertices of each edge in $ \pi_{n} $ are contained in exactly one maximal clique of $ G $.
Therefore the map $\phi$ is well-defined. Moreover, $ |\phi(\pi_{n})| \leq  |\pi_{n}|  $. 

Let $ G^{\prime} \coloneqq G\setminus \pi_{n} $. 
The restriction $ \lambda|_{E_{G^{\prime}}} $ is an MAT-labeling of $G'$ hence $ \mathcal{A}_{G^{\prime}} $ is MAT-free. 
Therefore the maximal exponents of $ \mathcal{A}_{G^{\prime}} $ are equal to $ n-1 $.
By part (1), $ \omega(G^{\prime}) = \omega(G)-1 $. 
Thus every largest clique of $ G $ contains the endvertices of at least one edge in $ \pi_{n} $. 
Hence every largest clique of $ G $ belongs to $ \phi(\pi_{n}) $ and $ |\phi(\pi_{n})| \geq  |\pi_{n}|  $. 
Thus $ \phi(\pi_{n}) $ is precisely the set consisting of the largest cliques. 
This completes the proof. 
\end{proof}

In general, a restriction of an MAT-labeling is not an MAT-labeling (e.g., when we restrict to any edge in $\pi_k$ with $k>1$). 
We show below that it is the case for restriction to certain subset. 
Recall that  $ \mathcal{P}_{G} $ denotes the clique intersection poset of a chordal graph $ G $ (\S \ref{sec:more}). 
\begin{lemma}\label{restriction to intersection of maximal cliques}
If $ \lambda $ is an MAT-labeling of a chordal graph $ G $, then the restriction $ \lambda|_{E_{G[X]}} $ is an MAT-labeling of the subgraph $ (V_{G}, E_{G[X]}) $ for any node $ X \in \mathcal{P}_{G} $. 
\end{lemma}
\begin{proof}
We proceed by induction on $ n = \omega(G)-1 $. 
Again it is easily seen that the assertion holds true for  $ n = 0,1 $.
Now suppose $ n \geq 2 $. 
First we consider the case $X=C$ where $C$ is a largest clique of $ G $. 
Note that by Proposition \ref{restriction to subset}, it suffices to prove that $ \lambda|_{E_{G[C]}} $ satisfies (ML\ref{definition MAT-labeling triangle}). 

Let $e=\{u,v\} \in\pi_n$ be the unique edge in $\pi_n$ whose endvertices are contained in $C$ (Proposition \ref{the number of maximal exponents = the number of largest cliques}).
The clique $ C $ is the union of two largest cliques $ C^{\prime} = C \setminus  \{v\}$ and $ C^{\prime\prime} = C \setminus  \{u\} $ of $ G^{\prime} = G\setminus\pi_{n} $. 
For $ W \in \{C,C^{\prime}, C^{\prime\prime}\} $, define $ \pi_{k}^{W} \coloneqq \lambda|_{E_{G[W]}}^{-1}(k) = \pi_{k} \cap E_{G[W]} $ for $ k \in [n]$. 
Note that $ \pi_{n}^{C} = \{e\} $ and $ \pi_{k}^{C} = \pi_{k}^{C^{\prime}} \cup \pi_{k}^{C^{\prime\prime}} $ for $ k \in [n-1] $ since the restrictions $ \lambda|_{E_{G[C^{\prime}]}} $ and $ \lambda|_{E_{G[C^{\prime\prime}]}} $ are MAT-labelings by the induction hypothesis. 

To show (ML\ref{definition MAT-labeling triangle}) of $ \lambda|_{E_{G[C]}} $, it suffices to prove that every edge in $ \pi_{k}^{C} $ forms at least $ k-1 $ triangles with edges in $ E_{k-1} \cap E_{G[C]} $ because every edge in $ \pi_{k}^{C} $ forms at most $ k-1 $ triangles with edges in $ E_{k-1} \cap E_{G[C]} $ by (ML\ref{definition MAT-labeling triangle}) of $ \lambda $. 
When $ k = n $, the edge $ e \in \pi_{n}^{C} $ forms $ n-1 $ triangles with edges in $E_{G[C]}\setminus\{e\} \subseteq E_{n-1} \cap E_{G[C]} $. 
We are left with $ k \in [n-1] $. 
Let $ f \in \pi_{k}^{C} $, then $ f \in \pi_{k}^{C^{\prime}} \cup \pi_{k}^{C^{\prime\prime}} $. 
Without loss of generality, we may assume $ f \in \pi_{k}^{C^{\prime}} $. 
Then $f$ forms $ k-1 $ triangles with edges in $ E_{k-1} \cap E_{G[C^{\prime}]} \subseteq E_{k-1} \cap E_{G[C]} $. 
Therefore $ \lambda|_{E_{G[C]}} $ satisfies (ML\ref{definition MAT-labeling triangle}) and hence $ \lambda|_{E_{G[C]}} $ is an MAT-labeling. 

Now we treat the case $ X \in \mathcal{P}_{G} $ is not a largest clique of $ G $. 
First we claim that $ E_{G[X]} \cap \pi_{n} = \varnothing $. 
If not, we can find an edge $ e \in E_{G[X]} \cap \pi_{n} $. 
Let $ C $ denote the largest clique such that $ e \in E_{G[C]} $. 
By the definition of $X$, there exists a maximal clique $ D $ of $G$ such that $X\subseteq  D \neq C$. 
Thus $ e \in E_{G[C\cap D]} $. 
Take a vertex $ c \in C\setminus D $. 
By the maximality of $ D $, there exists $ d \in D $ such that $\{c,d\} \notin E_G$. 
Then we obtain a chordless $4$-cycle  in $ G^{\prime} = G\setminus \pi_{n} $ formed by $ c,d $ and the endvertices of $ e $, which contradicts to the chordality of $ G^{\prime} $. 
Thus $ E_{G[X]} \cap \pi_{n} = \varnothing $. 

Let $C$ be a largest clique of $G$ and $e=\{u,v\} \in\pi_n$ be the unique edge in $\pi_n$ whose endvertices are contained in $C$. 
 Then $ C = C^{\prime} \cup C^{\prime\prime}  $ where $ C^{\prime} = C \setminus  \{v\}$ and $ C^{\prime\prime} = C \setminus  \{u\} $ are largest cliques of $ G^{\prime} = G\setminus\pi_{n} $. 
 By the above discussion, $e \notin E_{G[X]}$ hence either $u \notin X$ or $v \notin X$. 
 Thus either $X \subseteq C''$ or $X \subseteq C'$.
Note also that every maximal clique of $ G $ which is not largest is a maximal clique of $ G^{\prime} $. 
Therefore the node $ X $ is the intersection of some maximal cliques of $ G^{\prime} $. 
Hence $ X \in \mathcal{P}_{G^{\prime}} $ and by the induction hypothesis, $ \lambda|_{E_{G[X]}} = \left(\lambda|_{E_{G^{\prime}}}\right)|_{E_{G^{\prime}[X]}} $ is an MAT-labeling.
\end{proof}

We are ready to prove the main result of this subsection.
\begin{theorem}[MAT-freeness implies strong chordality]
\label{thm:main-1} 
If a simple graph $G$ admits an MAT-labeling, then $G$ is strongly chordal.  
\end{theorem}
\begin{proof}
Suppose that $ G $ is not strongly chordal. 
Note that $G$ is chordal by Proposition \ref{prop:MATF=MATL}.
Then $ G $ contains an $ n $-sun $ S_{n} $ (Definition \ref{def:strongly-chordal}) as an induced subgraph for some $ n \geq 3 $.

Let $ Z \coloneqq \{u_{1}, \dots, u_{n}\} $ be the central clique of $ S_{n} $, $ T_{i} \coloneqq \{u_{i}, v_{i}, u_{i+1}\} $ the vertex set of the triangle around $ Z $, and $ C_{i} $ a maximal clique of $ G $ containing $ T_{i} $ for $ i \in [n]$. 
Let $ G_{0} $ be the subgraph of $ G $ with vertex set $ V_{G_{0}} \coloneqq C_{1} \cup \dots \cup C_{n} $ and edge set $ E_{G_{0}} \coloneqq E_{G[C_{1}]} \cup \dots \cup E_{G[C_{n}]} $. 
By Lemmas \ref{restriction to intersection of maximal cliques} and \ref{restriction to the union}, $ G_{0} $ admits an MAT-labeling. 

Suppose $ n \geq 4 $. 
If $ \{u_{i}, u_{j}\} $ is an edge of $ G_{0} $, then $ \{u_{i}, u_{j}\} \in E_{G[C_{k}]} $ for some $ k \in [n]$. 
Therefore  both $ u_{i} $ and $ u_{j} $ are adjacent to $ v_{k} $. 
This implies $ \{u_{i}, u_{j}, v_{k}\} = T_{k} $ and hence $ j = i \pm 1 $. 
Thus the cycle in $ G_{0} $ consisting of edges $ \{u_{1}, u_{2}\}, \dots, \{u_{n-1}, u_{n}\}, \{u_{n}, u_{1}\} $ has no chords and its length is four or more. 
Therefore $ G_{0} $ is not chordal, which is a contradiction. 

Now we consider $ n = 3 $. 
Then $ Z = \{u_{1}, u_{2}, u_{3}\} $ is a clique of $ G_{0} $. 
Let $ K $ be a maximal clique of $ G_{0} $ containing $ Z $. 
Since $ u_{i} \not\in C_{i+1} $ for each $ i \in \{1,2,3\} $  ($C_4=C_1$), the clique $ K $ is neither $ C_{1}, C_{2} $, nor $ C_{3} $. 
Let $ X_{i} \coloneqq C_{i} \cap K \subsetneq K$ for each $ i \in \{1,2,3\} $. 
Then the restrictions $ \lambda|_{E_{G_{0}[K]}} $ and $ \lambda|_{E_{G_{0}[X_{i}]}}$ are MAT-labelings by Lemma \ref{restriction to intersection of maximal cliques}, where $ \lambda $ denotes an MAT-labeling of $ G_{0} $. 
Since $ E_{G_{0}} = E_{G[C_{1}]} \cup E_{G[C_{2}]} \cup E_{G[C_{3}]} $, we have $ E_{G_{0}[K]} = E_{G_{0}[X_{1}]} \cup E_{G_{0}[X_{2}]} \cup E_{G_{0}[X_{3}]} $. 
Therefore 
\begin{align*}
\max\left\{\lambda|_{E_{G_{0}[K]}}(e) \, \middle| \,  e \in E_{G_{0}[K]}\right\}
= \max\left\{ \lambda|_{E_{G_{0}[X_{i}]}}(e) \, \middle| \, i \in \{1,2,3\}, e \in E_{G_{0}[X_{i}]}\right\}. 
\end{align*}
Hence $ |K| = |X_{i}| $ for some $ i \in \{1,2,3\} $, which is a contradiction. 
\end{proof}

\begin{corollary}
The $ n $-sun $ S_{n} $ admits no MAT-labelings. 
\end{corollary}

%******************************************************************************** 

\section{Strong chordality implies MAT-freeness}
\subsection{MAT-simplicial vertices and MAT-perfect elimination orderings}
\label{subsec:MAT-S-PEO}
\quad
The proof of the ``if" part of our main Theorem \ref{thm:MAT-free-strong-chordal} (strong chordality implies MAT-freeness) requires more effort. 
We need a deeper understanding of the structure of graphs having MAT-labelings. 
In this subsecion, we develop a fundamental study on such graphs analogous to the theory of (strongly) chordal graphs by introducing MAT- versions of simplicial vertex and perfect elimination ordering.
\begin{definition}[MAT-simplicial vertices]
\label{definition MAT-simplicial}
Given an edge-labeled graph $ (G,\lambda) $, a vertex $ v \in V_{G} $ is said to be \textbf{MAT-simplicial} if the following conditions hold. 
\begin{enumerate}[(MS1)]
\item\label{definition MAT-simplicial 1} $ v $ is a simplicial vertex of $ G $, that is, its neighborhood $ N_{G}(v) $ is a clique of $ G $.
\item\label{definition MAT-simplicial 2} $ \Set{\lambda(\{u,v\}) \in \mathbb{Z}_{>0} | u \in N_{G}(v) } = \{1,2, \dots, \deg_{G}(v)\} $, where $ \deg_{G}(v) =|N_{G}(v) |$ denotes the degree  of $ v $ in $ G $. 
\item\label{definition MAT-simplicial 3} For any distinct vertices $ u_{1}, u_{2} \in N_{G}(v) $, $ \lambda(\{u_{1}, u_{2}\}) < \max\{\lambda(\{u_{1}, v\}), \lambda(\{u_{2}, v\})\} $. 
\end{enumerate}
\end{definition}

Next we show the existence of MAT-simplicial vertices in the graphs having MAT-labelings.

\begin{lemma}
\label{lem:MATS-existence}
Let $ (G, \lambda) $ be an edge-labeled graph such that $ |V_{G}| \geq 2 $ and $ \lambda $ is an MAT-labeling of $ G $. 
\begin{enumerate}[(1)]
\item\label{MATS-existence 1} If $ G = K_{\ell} $ is a complete graph, then the endvertices of the edge with maximal label are MAT-simplicial. 
\item\label{MATS-existence 2} 
If $ G $ is noncomplete, then $ (G, \lambda) $ has two nonadjacent MAT-simplicial vertices. 
\end{enumerate}
\end{lemma}
\begin{proof}
First we prove part (\ref{MATS-existence 1}). 
Let $ e_{0} = \{u_{0}, v_{0}\} \in E_{G} $ be the edge with maximal label. 
It suffices to prove that $ v_{0} $ is MAT-simplicial. 
First, (MS\ref{definition MAT-simplicial 1}) is clear. 
Next we show (MS\ref{definition MAT-simplicial 2}). 
Note that $ \lambda|_{E_{G\setminus e_0}} $ is an MAT-labeling.
Then by Lemma \ref{restriction to intersection of maximal cliques}, the labelings $ \lambda|_{E_{G[C]}} $ and $ \lambda|_{E_{G[X]}} $ are MAT-labelings, where $ C \coloneqq V_{G}\setminus\{u_{0}\} $ and $ X \coloneqq V_{G} \setminus \{u_{0}, v_{0}\} $. 
Comparing the exponents of $ \mathcal{A}_{G[C]} $ and $ \mathcal{A}_{G[X]} $, we have $ \Set{\lambda(\{v, v_{0}\}) \in \mathbb{Z}_{>0} | v \in X } = [\ell -2] $. 
Since $ \lambda(e_{0}) = \ell -1 $, $ \Set{\lambda(\{v, v_{0}\}) \in \mathbb{Z}_{>0} | v \in N_{G}(v_{0}) } = [\ell -1] $.
Thus (MS\ref{definition MAT-simplicial 2}) is satisfied. 

Lastly,  we show (MS\ref{definition MAT-simplicial 3}). 
Let $ u,v \in X $ and write $ \lambda(\{u,v\}) = k $. 
We want to show $ k<\max\{\lambda(\{u,v_{0}\}), \lambda(\{v,v_{0}\})\} $. 
Since both $ \lambda|_{E_{G[C]}} $ and $ \lambda|_{E_{G[X]}} $ are MAT-labelings, (ML\ref{definition MAT-labeling triangle}) implies $ k \leq \max\{\lambda(\{u,v_{0}\}), \lambda(\{v,v_{0}\})\} $. 
If the equality happens, then it contradicts to (ML\ref{definition MAT-labeling closure}) of $ \lambda|_{E_{G[C]}} $. 
Therefore $ k < \max\{\lambda(\{u,v_{0}\}), \lambda(\{v,v_{0}\})\} $. 
Moreover, $ \max\Set{\lambda(\{v,u_{0}\}) | v \in X} = \ell -2 < \ell - 1 = \lambda(e_{0}) $, since $ \lambda|_{E_{G\setminus v_{0}}} $ is also an MAT-labeling by Lemma \ref{restriction to intersection of maximal cliques}. 
Therefore (MS\ref{definition MAT-simplicial 3}) holds. 
Thus $ v_0 $ is MAT-simplicial. 

Now we prove part (\ref{MATS-existence 2}). 
 We proceed induction on $ \ell = |V_{G}| $. 
If $ \ell=2 $, then the assertion holds trivially. 
Suppose $ \ell \geq 3 $. 
We may assume that $ G $ is connected. 
Let $ a,b \in V_{G} $ be nonadjacent vertices and $ S $ a minimal $ (a,b) $-separator. 
Then $ S $ is a clique by Theorem \ref{Dirac minimal vertex separator} and $ S \in \mathcal{P}_{G} $ by Remark \ref{Ho and Lee}. 
Hence $ \lambda|_{E_{G[S]}} $ is an MAT-labeling of $ G[S] $ by Lemma \ref{restriction to intersection of maximal cliques}.

Let $ A $ be the vertex set of the connected component containing $ a $ of $ G\setminus S $ and $ B \coloneqq V_{G} \setminus (A \cup S) $. 
We will show that $ \lambda|_{E_{G[A\cup S]}} $ and $ \lambda|_{E_{G[B\cup S]}} $ are MAT-labelings. 
By Proposition \ref{restriction to subset}, it suffices to show (ML\ref{definition MAT-labeling triangle}). 
Let $ e \in \pi_{k} \cap E_{G[A \cup S]} $. 
Then $ e $ forms at most $ k-1 $ triangles with edges in $ E_{k-1} \cap E_{G[A \cup S]} $ since $ \lambda $ is an MAT-labeling. 
When $ e \in E_{G[S]} $, $ e $ forms exactly $ k-1 $ triangles with edges in $ E_{k-1} \cap E_{G[S]} $. 
Therefore $ e $ forms exactly $ k-1 $ triangles with edges in $ E_{k-1} \cap E_{G[A \cup S]} $. 
Suppose that at least one endvertex of $ e $ belongs to $ A $. 
Then $ e $ cannot form a triangle with a vertex in $ B $. 
Hence $ e $ forms exactly $ k-1 $ triangles with edges in $ E_{k-1} \cap E_{G[A \cup S]} $. 
Thus $ \lambda|_{E_{G[A\cup S]}} $ is an MAT-labeling.
By a similar way, one can prove that $ \lambda|_{E_{G[B\cup S]}} $ is an MAT-labeling.

Next we show that $ A \setminus S $ contains an MAT-simplicial vertex of $ G $. 
Note that if $ v \in A \setminus S $ is MAT-simplicial in $ G[A\cup S] $, then $ v $ is also MAT-simplicial in $ G $ since $ N_{G}(v) \subseteq A \cup S $. 
If $ G[A \cup S] $ is a complete graph, then the endvertices of the edge $ e_{0} $ in $ G[A \cup S] $ with maximal label is MAT-simplicial in $ G[A \cup S] $ by  part \eqref{MATS-existence 1}. 
Since $ \lambda|_{E_{G[S]}} $ is an MAT-labeling, at least one endvertex of $ e_{0} $ belongs to $ A \setminus S $ by Proposition \ref{MAT-labeling on complete graph}, which is a desired MAT-simplicial vertex. 
If $ G[A \cup S] $ is not a complete graph, then by the induction hypothesis $ G[A \cup S] $ has two nonadjacent MAT-simplicial vertices. 
At least one of them belongs to $ A \setminus S $ since $ S $ is a clique. 
Thus $ A \setminus S $ contains an MAT-simplicial vertex of $ G $. 

Similarly, $ B \setminus S $ contains an MAT-simplicial vertex of $ G $. 
Therefore $ G $ has two nonadjacent MAT-simplicial vertices. 
\end{proof}

The following is a first important property of MAT-simplicial vertices.
\begin{proposition}\label{MAT-simplicial}
Let $ (G, \lambda) $ be an edge-labeled graph with $ |V_{G}| \geq 2 $. 
Suppose that $ v \in V_{G} $ is an MAT-simplicial vertex of $ (G,\lambda) $. 
The following are equivalent. 
\begin{enumerate}[(1)]
\item $ \lambda $ is an MAT-labeling of $ G $. 
\item $ \lambda|_{E_{G\setminus v}} $ is an MAT-labeling of $ G\setminus v $. 
\end{enumerate}
\end{proposition}
\begin{proof}
First we prove $ (1) \Rightarrow (2) $. 
Let $ \pi_{k}^{\prime} \coloneqq (\lambda|_{E_{G\setminus v}})^{-1}(k) =\pi_k \cap E_{G\setminus v}$ and $ E_{k-1}^{\prime} \coloneqq \pi^{\prime}_{1} \sqcup \dots \sqcup \pi^{\prime}_{k-1} $ for $ k \in \mathbb{Z}_{>0} $. 
By Proposition  \ref{restriction to subset}, we only need to prove (ML\ref{definition MAT-labeling triangle}) of $ \lambda|_{E_{G\setminus v}} $.

Let $ e \in \pi_{k}^{\prime} $. 
Since $ e \in \pi_{k}^{\prime} \subseteq \pi_{k} $, $ e $ forms exactly $ k-1 $ triangles with edges in $ E_{k-1} $. 
These triangles do not contain the vertex $ v $ because by (MS\ref{definition MAT-simplicial 3}) the number of edges incident to $ v $ with label less than $ k $ is at most $ 1 $. 
Therefore $ e $ forms exactly $ k-1 $ triangles with edges in $ E_{k-1}^{\prime} $. 
Thus $ \lambda|_{E_{G\setminus v}} $ is an MAT-labeling of $ G\setminus v $. 

Next we prove $ (2) \Rightarrow (1) $. 
Let $ k \in \mathbb{Z}_{>0} $.
By  (MS\ref{definition MAT-simplicial 2}), $ v $ is a leaf or  an isolated vertex of $ \pi_{k} $. 
Moreover, since $ \pi_{k}^{\prime} $ is a forest, $ \pi_{k} $ is a forest. 
This shows (ML\ref{definition MAT-labeling forest}). 

To show (ML\ref{definition MAT-labeling closure}), suppose $ \cl_{G}(\pi_{k}) \cap E_{k-1} \neq \varnothing $ and take $ e \in \cl_{G}(\pi_{k}) \cap E_{k-1} $. 
Then there exists a cycle $ C $ in $G$ such that $ e \in C $ and $ C \setminus e \subseteq \pi_{k} $. 
If $ e $ is not incident to $ v $ (in particular, $v$ is not a vertex of $C$), then $ e \in E_{k-1}^{\prime} $ and $ C \setminus e \subseteq \pi_{k}^{\prime} $. 
Therefore $ e \in \cl_{G\setminus v}(\pi_{k}^{\prime}) \cap E_{k-1}^{\prime} = \varnothing$, a contradiction. 
Hence $ e $ is incident to $ v $, and $ C $ contains an edge $ \{v,w\} $ with $ \lambda(\{v,w\}) = k $. 
Write $ e = \{u,v\} $. 
Then $ \{u,w\} \in E_{G} $ by (MS\ref{definition MAT-simplicial 1}) and $ \lambda(\{u,w\}) < \max\{\lambda(\{u,v\}), \lambda(\{v,w\})\} = k $ by (MS\ref{definition MAT-simplicial 3}). 
Hence $ \{u,w\} \in E_{k-1}^{\prime} $ (in particular, $C$ has length at least $4$).
Moreover,  $ \{u,w\}$ and $ C \setminus\{\{v,w\}, e\} \subseteq \pi_{k}^{\prime} $ form a cycle and hence $ \{u,w\} \in \cl_{G\setminus v}(\pi_{k}^{\prime}) \cap E_{k-1}^{\prime}= \varnothing$, a contradiction. 

Finally, we prove (ML\ref{definition MAT-labeling triangle}). 
Let $ e \in \pi_{k} $.
If $ e \in \pi_{k}^{\prime} $ (i.e, $ e $ is not incident to $ v $), then $ e $ forms exactly $ k-1 $ triangles with some edges in $ E_{k-1}^{\prime} $. 
If one endvertex of $ e $ is not adjacent to $ v $, then $ e $ and $v$ cannot form a triangle. 
If both endvertices of $ e $ are adjacent to $ v $, then at least one edge of the triangle containing $ e $ and $ v $ has label greater than $ k $ by (MS\ref{definition MAT-simplicial 3}). 
In either case, $ e $ forms exactly $ k-1 $ triangles with some edges in $ E_{k-1} $. 
Now consider  $ e \in \pi_{k} \setminus \pi_{k}^{\prime} $ (i.e., $ e $ is incident to $ v $). 
By  (MS\ref{definition MAT-simplicial 2}), we can index the elements of $ N_{G}(v)$  as $\{u_{1}, \dots, u_{d}\} $  where $k\le d=\deg_G(v)$ so that $ e = \{u_{k}, v\} $ and $ \lambda(\{u_{i}, v\}) = i $ for every $ i \in [d] $. 
Hence $ e $ forms exactly $ k-1 $ triangles given by $ \{u_{i}, u_{k}, v\} $ for $ i \in [k-1] $ with some edges in $ E_{k-1} $.  
Thus $ \lambda $ is an MAT-labeling.  
\end{proof}

\begin{definition}[MAT-PEO]
Given an edge-labeled graph $ (G,\lambda) $, an ordering $ (v_{1}, \dots, v_{\ell}) $ of $ G $ is said to be a \textbf{MAT-perfect elimination ordering (MAT-PEO)} of $ (G,\lambda) $ if $ v_{i} $ is MAT-simplicial in $ (G_{i}, \lambda_{i}) $ for each $ i \in [\ell] $, where $ G_{i} \coloneqq G[\{v_{1}, \dots, v_{i}\}] $ and $ \lambda_{i} \coloneqq \lambda|_{E_{G_{i}}} $. 
\end{definition}

In particular, any MAT-PEO is a PEO.
The theorem below exhibits a strong connection between MAT-labelings and MAT-PEOs which can be seen as an analogue of Theorem \ref{thm:chordal-PEO}.
\begin{theorem}[MAT-labelings and MAT-PEOs]
\label{MAT-labeling ordering}
Given an edge-labeled graph $ (G, \lambda) $, the following are equivalent. 
\begin{enumerate}[(1)]
\item $ \lambda $ is an MAT-labeling of $ G $. 
\item There exists an MAT-PEO of $ (G, \lambda) $. 
\end{enumerate}
\end{theorem}
\begin{proof}
Both implications can be proved by induction on $ \ell = |V_{G}| $.
The implication $ (2) \Rightarrow (1) $ is easy thanks to Proposition \ref{MAT-simplicial}. 
Let us show the converse (which is also not hard). 
When $ \ell = 1 $, the assertion is true.
Suppose $ \ell \geq 2 $. 
By Lemma \ref{lem:MATS-existence}, there exists an MAT-simplicial vertex $ v_{\ell} $ of $ (G, \lambda) $. 
Then $ \lambda|_{E_{G\setminus v_{\ell}}} $ is an MAT-labeling of $ G\setminus v_{\ell} $ by Proposition \ref{MAT-simplicial}. 
By the induction hypothesis, $ (G\setminus v_{\ell}, \lambda|_{E_{G\setminus v_{\ell}}}) $ has an MAT-PEO $ (v_{1}, \dots, v_{\ell -1}) $. 
Thus $ (v_{1}, \dots, v_{\ell}) $ is an MAT-PEO of $ (G,\lambda) $. 
\end{proof}

We complete this subsection by giving two lemmas on (extensions of) MAT-labelings and MAT-PEOs of complete graphs. 
MAT-labelings of complete graphs will play a crucial role in the next subsection. 

\begin{lemma}\label{MAT-L-PEO of complete graph}
Let $ G = K_{\ell} $ be a complete graph and $ W \subseteq V_{G} $. 
\begin{enumerate}[(1)]
\item\label{MAT-L-PEO of complete graph 1} 
Let $ \lambda $ be an MAT-labeling of $ G$. If $ (v_{1}, \dots, v_{r}) $ is an MAT-PEO  of $ (G[W], \lambda|_{E_{G[W]}}) $, then $ (v_{1}, \dots, v_{r}) $ can be extended to an MAT-PEO $ (v_{1}, \dots, v_{r}, \dots, v_{\ell}) $ of $ (G,\lambda) $. 
\item\label{MAT-L-PEO of complete graph 2} 

If $ \lambda_{W} $ is an MAT-labeling of $ G[W] $, then $ \lambda_{W} $ can be extended to an MAT-labeling of $ G $. 

\end{enumerate}
As a consequence, a complete graph always has an MAT-labeling (and an MAT-PEO) which can be constructed inductively from any vertex of the graph.
\end{lemma}
\begin{proof}
(\ref{MAT-L-PEO of complete graph 1}) 
We proceed by induction on $ \ell = |V_{G}| $. 
When $ \ell = 1 $, we have nothing to prove. 
Now suppose $ \ell \geq 2 $. 
If $ W = V_{G} $, the assertion holds trivially. 
Suppose $ W \subsetneq V_{G} $. 
Let $ e_{0}\in E_G $ be the edge with maximal label. 
Since $ \max\Set{\lambda(e) \in \mathbb{Z}_{>0} | e \in E_{G[W]}} < \lambda(e_{0}) = \ell-1 $, at least one endvertex of $ e_{0} $, say $ v_{\ell} $  does not belong to $ W $. 
By Lemma \ref{lem:MATS-existence}(\ref{MATS-existence 1}), $ v_{\ell} $ is MAT-simplicial in $ G $. 
By the induction hypothesis, there exists an MAT-PEO $ (v_{1}, \dots, v_{r}, \dots, v_{\ell-1}) $ of $ G\setminus v_{\ell} $. 
Hence $ (v_{1}, \dots, v_{\ell}) $ is a desired ordering. 

(\ref{MAT-L-PEO of complete graph 2}) 
Without loss of generality, we may assume $ \ell \geq 2 $ and $|W| = \ell - 1 $. 
By Theorem \ref{MAT-labeling ordering}, there exists an MAT-PEO $ (v_{1}, \dots, v_{\ell-1}) $ of $ (G[W], \lambda_{W}) $.
Let $ v_{\ell} $ denote the vertex in $ V_{G}\setminus W $. 
We define a labeling $ \lambda $ of $ G $ by 
\begin{align*}
\lambda(e) \coloneqq
\begin{cases}
\lambda_{W}(e) & \text{ if } e \in E_{G[W]}; \\
i &  \text{ if } e = \{v_{i}, v_{\ell}\} \text{ for } i \in  [\ell -1]. 
\end{cases}
\end{align*}

We will show that $ v_{\ell} $ is MAT-simplicial in $ (G,\lambda) $. 
Firstly, (MS\ref{definition MAT-simplicial 1}) is clear since $ G $ is complete. 
Secondly, we have $ \Set{\lambda(\{v_{i}, v_{\ell}\}) | i \in [\ell-1]} = [\ell-1] $ and hence (MS\ref{definition MAT-simplicial 2}) holds. 
Thirdly, for $ 1\le i < j < \ell $, we have $ \lambda(\{v_{i}, v_{j}\}) \leq j-1 < j = \lambda(\{v_{j}, v_{\ell}\}) $, which shows (MS\ref{definition MAT-simplicial 3}). 
Therefore $ v_{\ell} $ is MAT-simplicial in $ (G,\lambda) $. 
Thus $ (v_{1}, \dots, v_{\ell} )$ is an MAT-PEO in $ (G, \lambda) $ and $ \lambda $ is an MAT-labeling by Proposition \ref{MAT-simplicial}.
\end{proof}

\begin{lemma}\label{merge}
Let $ G = K_{\ell} $. 
Suppose that $ V_{G} = A \cup B $ and there exist MAT-labelings $ \lambda_{A}, \lambda_{B}, \lambda_{A \cap B} $ of $ G[A], G[B], G[A\cap B] $, respectively such that $ \lambda_{A}|_{E_{G[A\cap B]}} = \lambda_{B}|_{E_{G[A\cap B]}} = \lambda_{A \cap B} $. 
Then there exists an MAT-labeling $ \lambda $ of $ G $ such that $ \lambda|_{E_{G[A]}} = \lambda_{A} $ and $ \lambda|_{E_{G[B]}} = \lambda_{B} $. 
\end{lemma}
\begin{proof}
By Lemma \ref{MAT-L-PEO of complete graph}(\ref{MAT-L-PEO of complete graph 1}), there exists an MAT-PEO $ (a_{1}, \dots, a_{p}) $ of $ G[A \cap B] $ and its extensions $ (a_{1}, \dots, a_{p}, a_{p+1}, \dots, a_{p+q}) $ of $ G[A] $ and $ (a_{1}, \dots, a_{p}, b_{1}, \dots, b_{r}) $ of $ G[B] $ (where $p+q+r=\ell$). 
Define a labeling $ \lambda \colon E_{G} \to \mathbb{Z}_{>0} $ by 
\begin{align*}
\lambda(e) \coloneqq \begin{cases}
\lambda_{A}(e) & \text{ if } e \in E_{G[A]}; \\
\lambda_{B}(e) & \text{ if } e \in E_{G[B]}; \\
p + i + j - 1 & \text{ if } e = \{a_{p+i}, b_{j}\} \quad (i \in [q], \,  j \in[  r]). 
\end{cases}
\end{align*}

We claim that $ \lambda $ is a desired MAT-labeling of $ G $ by induction on $ r $. 
If $ r = 0 $, then $ \lambda = \lambda_{A} $ and hence the claim holds. 
Suppose $ r \geq 1 $. 
We will prove that $ b_{r} $ is MAT-simplicial in $ (G, \lambda) $. 
The condition (MS\ref{definition MAT-simplicial 1}) is clear. 
Since $ b_{r} $ is MAT-simplicial in $ (G[B], \lambda_{B}) $, 
\begin{align*}
\Set{\lambda(\{a_{i}, b_{r}\}) | i \in [p ] } \cup \Set{\lambda(\{b_{j}, b_{r}\}) | j \in[  r-1]} =[ p+r-1]. 
\end{align*}
By the definition of $ \lambda $ we have $ \lambda(\{a_{p+i}, b_{r}\}) = p+i+r-1 \ (i \in [q]) $. 
Therefore $ \Set{\lambda(\{v,b_{r}\}) | v \in N_{G}(b_{r})} = [\ell - 1] $ and hence (MS\ref{definition MAT-simplicial 2}) holds. 

Next we show (MS\ref{definition MAT-simplicial 3}), i.e., $ \lambda(\{u,v\}) < \max\{\lambda(\{u,b_{r}\}), \lambda(\{v,b_{r}\})\} $ for any distinct vertices $ u,v \in N_{G}(b_{r}) $. 
It is clear when $ u,v \in B $ since $ b_{r} $ is MAT-simplicial in $ (G[B], \lambda_{B}) $. 
Consider the case $ u=a_{p+i} \in A\setminus B, v=a_{j} \in A \cap B$ for some $ i,j $ with $ p+i > j $. 
Then $ \lambda(\{a_{p+i}, a_{j}\}) \leq p+i-1 < p+i+r-1 = \lambda(\{a_{p+i}, b_{r}\}) $ since $ a_{p+i} $ is MAT-simplicial in $ (G, \lambda|_{E_{G[\{a_{1}, \dots, a_{p+i}\}]}}) $. 
Now consider $ u=a_{p+i} \in A\setminus B, v = b_{j} \in B\setminus A $ for some $ i,j $ with $ 1 \leq i \leq q $ and $ 1 \leq j < r $. 
Then $ \lambda(\{a_{p+i}, b_{j}\}) = p+i+j-1 < p+i+r-1 = \lambda(\{a_{p+i}, b_{r}\}) $. 
Thus (MS\ref{definition MAT-simplicial 3}) holds and $ b_{r} $ is an MAT-simplicial vertex of $ (G, \lambda) $. 

By the induction hypothesis, $ \lambda|_{E_{G\setminus b_{r}}} $ is an MAT-labeling. 
Using Proposition \ref{MAT-simplicial}, we conclude that $ \lambda $ is an MAT-labeling. 
\end{proof}

%******************************************************************************** 
\subsection{Proof of the implication ``strongly chordal  $\Rightarrow$ MAT-free"}
\label{sub:if-part}
\quad
In this subsection we prove the ``if" part of our main Theorem \ref{thm:MAT-free-strong-chordal} (strong chordality implies MAT-freeness). 
To find an MAT-labeling for a given strongly chordal graph, our strategy is to find compatible MAT-labelings of the subgraphs induced by all maximal cliques, then combine the constructions by the following ``gluing trick".
\begin{theorem}[``Gluing trick"]
\label{glue}
Let $ G $ be a simple graph and suppose that $ V_{G} = A \cup B$, $E_{G} = E_{G[A]} \cup E_{G[B]} $, and $ A \cap B $ is a clique. 
Assume that there exist MAT-labelings $ \lambda_{A}, \lambda_{B}, \lambda_{A \cap B} $ of $ G[A], G[B], G[A\cap B] $, respectively  such that $ \lambda_{A}|_{E_{G[A\cap B]}} = \lambda_{B}|_{E_{G[A\cap B]}} =\lambda_{A\cap B}$. 
Define $ \lambda :=\lambda_{A} \cup \lambda_{B} \colon E_{G} \to \mathbb{Z}_{>0} $ by $\lambda|_A = \lambda_{A}, \lambda|_B = \lambda_{B}$, i.e.,
\begin{align*}
\lambda(e) \coloneqq \begin{cases}
\lambda_{A}(e), & \text{ if } e \in E_{G[A]}, \\
\lambda_{B}(e), & \text{ if } e \in E_{G[B]}. 
\end{cases}
\end{align*}
\end{theorem}
\noindent
Then $ \lambda $ is an MAT-labeling of $ G $.
 
\begin{proof}
We proceed by induction on $ \ell = |V_{G}| $.
When $ \ell \leq 2 $ the assertion is trivial. 
Suppose $ \ell \geq 3 $. 
We may assume $ A \setminus  B \neq \varnothing $. 

We claim that $ G[A] $ has an MAT-simplicial vertex in $ A \setminus B $. 
First consider the case $ A $ is a clique. 
Then any MAT-PEO of $ (G[A\cap B], \lambda_{A\cap B}) $ is extended to an MAT-PEO of $ (G[A], \lambda_{A}) $ by Lemma \ref{MAT-L-PEO of complete graph}(\ref{MAT-L-PEO of complete graph 1}), which shows the claim. 
Next suppose that $ A $ is not a clique. 
Then $ G[A] $ has two nonadjacent MAT-simplicial vertices by Lemma \ref{lem:MATS-existence}(\ref{MATS-existence 2}). 
At least one of them belongs to $ A \setminus B $ since $ A \cap B $ is a clique. 
Thus, in either case, $ G[A] $ has an MAT-simplicial vertex, say  $ v_{\ell} $ in $A \setminus B$.
Note that $ v_{\ell} $ is MAT-simplicial also in $ (G, \lambda) $ since $ N_{G}(v_{\ell}) \subseteq E_{G[A]} $. 

By Lemma \ref{MAT-simplicial}, $ \lambda_A|_{E_{G[A]\setminus v_{\ell}}} $ is an MAT-labeling. 
Consider the graph $ G\setminus v_{\ell} $ with the decomposition $ V_{G\setminus v_{\ell}} = (A\setminus \{v_{\ell}\}) \cup B $. 
By the induction hypothesis, we have that $ \lambda|_{E_{G \setminus v_{\ell}}} $ is an MAT-labeling. 
Using Lemma \ref{MAT-simplicial} again, we conclude that $ \lambda $ is an MAT-labeling of $G$. 
\end{proof}

The lemma below describes an important property (of antichains) of the clique intersection poset $ \mathcal{P}_{G} $.
\begin{lemma}\label{leaf}
Let $ G $ be a strongly chordal graph and let $ T \subseteq \mathcal{P}_{G} $ be an antichain with $ |T| \geq 2 $.
Then there exist distinct $ X_{0}, Y_{0} \in T $ such that $ X_{0} \cap Y_{0} \supseteq X_{0} \cap Y $ for all $ Y \in T \setminus \{X_{0}\} $. 
\end{lemma}
\begin{proof}
Let $ Q $ be the subposet of $ \mathcal{P}_{G} $ induced by $  \Set{X \cap Y \in \mathcal{P}_{G} | X,Y \in T}  \supseteq T$. 
We will show that there exists a node $ X_{0} \in T $ such that $ X_{0} $ is a leaf of the Hasse diagram $\mcH(Q)$ of $ Q $. 

Consider induced subposets of $ Q $ whose Hasse diagrams have the following form: 
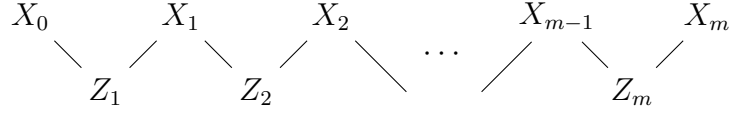
\begin{center}
\begin{tikzpicture}[]
\node (X0) at (0,1) {$ X_{0} $};
\node (Z1) at (1,0) {$ Z_{1} $};
\node (X1) at (2,1) {$ X_{1} $};
\node (Z2) at (3,0) {$ Z_{2} $};
\node (X2) at (4,1) {$ X_{2} $};
\node (D) at (5.5,0.5) {$ \cdots $};
\node (Xm-1) at (7,1) {$ X_{m-1} $};
\node (Zm) at (8,0) {$ Z_{m} $};
\node (Xm) at (9,1) {$ X_{m} $};
\draw (X0)--(Z1)--(X1)--(Z2)--(X2)--(5,0);
\draw (6,0)--(Xm-1)--(Zm)--(Xm);
\end{tikzpicture}
\end{center}
where $ m \geq 0 $ and $ X_{i} \in T $ for all $ i \in \{0, \dots, m\} $.
Let $F= \{X_{0}, Z_{1}, \dots, Z_{m},X_{m}\} \subseteq Q $ be a poset of the form above such that $ m $ is maximum. 
Since the Hasse diagram does not change when we replace $ Z_{i} $ by $ X_{i-1} \cap X_{i} $ for each $ i \in \{1, \dots, m\} $, we may assume $ Z_{i} = X_{i-1} \cap X_{i} $. 

If $ X_{0} $ is not a leaf of $\mcH(Q)$, then there exists a node $ Z^{\prime} \in Q $ such that in $\mcH(Q)$,  $ Z^{\prime} $ is covered by $ X_{0} $ and the pair $ \{Z^{\prime}, Z_{1}\} $ is incomparable. 
Since every element in $ Q\setminus T $ is the intersection of some elements in $ T $, there exists a node $ X^{\prime} \in T $ such that $ X^{\prime} \supsetneq Z^{\prime} $ and $ X^{\prime} \neq X_{0} $. 
Similarly, we may assume $ Z^{\prime} = X^{\prime} \cap X_{0} $. 

Since $ m $ is maximum, either $ X^{\prime} $ contains some $ Z_{j} \in F$, or $ Z^{\prime} $ is contained in some  $ X_{j} \in F \setminus \{X_0\}$. 
In either case, we obtain an induced subposet of $Q$ whose Hasse diagram has the form: 
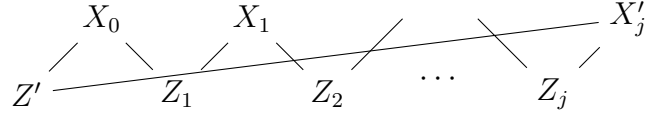
\begin{center}
\begin{tikzpicture}[]
\node (Zp) at (0,0) {$ Z^{\prime} $};
\node (X0) at (1,1) {$ X_{0} $};
\node (Z1) at (2,0) {$ Z_{1} $};
\node (X1) at (3,1) {$ X_{1} $};
\node (Z2) at (4,0) {$ Z_{2} $};
\node (D) at (5.5,0.2) {$ \cdots $};
\node (Zj) at (7,0) {$ Z_{j} $};
\node (Xj) at (8,1) {$ X_{j}^{\prime} $};
\draw (Zp)--(X0)--(Z1)--(X1)--(Z2)--(5,1);
\draw (6,1)--(Zj)--(Xj)--(Zp); 
\end{tikzpicture}
\end{center}
where $ X_{j}^{\prime} $ denotes $ X_{j} $ or $ X^{\prime} $.

Since $ G $ is strongly chordal, $ \mathcal{P}_{G} $ is $ k $-crown-free for all $ k \geq 3 $ by Theorem \ref{Nevries-Rosenke}. 
Hence $ j = 1 $. 
Thus this leads to one of the following induced subposets: 
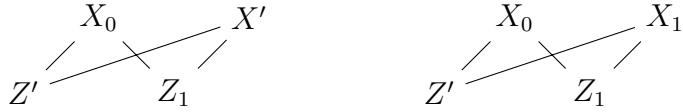
\begin{center}
\begin{tikzpicture}[baseline=10]
\node (Zp) at (0,0) {$ Z^{\prime} $};
\node (X0) at (1,1) {$ X_{0} $};
\node (Z1) at (2,0) {$ Z_{1} $};
\node (Xp) at (3,1) {$ X^{\prime} $};
\draw (Zp)--(X0)--(Z1)--(Xp)--(Zp);
\end{tikzpicture}
\hspace{7mm} \hspace{7mm}
\begin{tikzpicture}[baseline=10]
\node (Zp) at (0,0) {$ Z^{\prime} $};
\node (X0) at (1,1) {$ X_{0} $};
\node (Z1) at (2,0) {$ Z_{1} $};
\node (X1) at (3,1) {$ X_{1} $};
\draw (Zp)--(X0)--(Z1)--(X1)--(Zp);
\end{tikzpicture}
\end{center}

If the first case occurs, then $ Z_{1} \subseteq X_{0} \cap X^{\prime} = Z^{\prime} $, which contradicts the incomparability of $ Z_{1} $ and $ Z^{\prime} $. 
When the second case occurs, $ Z^{\prime} \subseteq X_{0} \cap X_{1} = Z_{1} $, a contradiction again. 
In summary, $ X_{0} $ is a leaf of $\mcH(Q)$. 

Now let $ Y_{0} \cap Y_{1} \in Q$  for $Y_{0}, Y_{1} \in T$ denote a unique node in $ Q $ that is covered by $ X_{0} $. 
We may assume that one of $Y_{0}, Y_{1}$ is not $X_0$, say $ Y_{0} \neq X_{0} $.
Then $ X_{0} \supsetneq X_{0} \cap Y_{0} \supseteq X_{0} \cap Y_{0} \cap Y_{1} = Y_{0} \cap Y_{1} $. 
Hence $ X_{0} \cap Y_{0} = Y_{0} \cap Y_{1} $. 

Finally let $ Y \in T\setminus\{X_{0}\} $. 
Then $ X_{0} \cap Y \in Q $ and $ X_{0} \cap Y \subsetneq X_{0} $. 
Since every path from $ X_{0} \cap Y $ to $ X_{0} $ passes through $   Y_{0} \cap Y_{1} $ in $\mcH(Q)$, we must have $ X_{0} \cap Y \subseteq X_{0} \cap Y_{0} $. 
\end{proof}

The lemma below shows the existence of constituent MAT-labelings compatible with the ``gluing trick".
\begin{lemma}\label{node MAT-labeling}
Let $ G $ be a strongly chordal graph. 
Then there exists a family $\F(\mathcal{P}_{G})= \{\lambda_{X}\}_{X \in \mathcal{P}_{G}} $ consisting of MAT-labelings $ \lambda_{X} $ of $ G[X] $ such that $\F(\mathcal{P}_{G})$ is closed under inclusion, i.e., $ \lambda_{X}|_{E_{G[Y]}} = \lambda_{Y} $ whenever $ X \supseteq Y $. 
\end{lemma}
\begin{proof}
We define the rank of a node $ X \in \mathcal{P}_{G} $ as the length of a maximum chain connecting $ X $ and $ \hat{0} $.
Let $ \mathcal{P}_{G}^{r} $ denote the set consisting of the nodes of rank at most $ r $. 
We will show by induction on $ r $ that there exists a family $\F( \mathcal{P}_{G}^{r})= \{\lambda_{X}\}_{X \in \mathcal{P}_{G}^{r}} $ consisting of MAT-labelings $ \lambda_{X} $ of $ G[X] $ such that $ \lambda_{X}|_{E_{G[Y]}} = \lambda_{Y} $ whenever $ X \supseteq Y $. 

When $ r = 0 $, $ \mathcal{P}_{G}^{0} = \{\hat{0}\} $. 
Since $ G[\hat{0}] $ is a complete graph (or null graph), there exists an MAT-labeling of $ G[\hat{0}] $ by Lemma \ref{MAT-L-PEO of complete graph}. 

Now suppose $ r > 0 $.
Then by the induction hypothesis there exists a family $\F( \mathcal{P}_{G}^{r-1})= \{\lambda_{Y}\}_{Y \in \mathcal{P}_{G}^{r-1}} $ consisting of MAT-labelings $ \lambda_{Y} $ of $ G[Y] $ such that $ \lambda_{Y_{1}}|_{E_{G[Y_{2}]}} = \lambda_{Y_{2}} $ whenever $ Y_{1} \supseteq Y_{2} $. 
We prove the following claim.
\begin{claim} 
\label{cl:construction} 
 Let $ X \in \mathcal{P}_{G}^{r} \setminus \mathcal{P}_{G}^{r-1} $ and $ T \subseteq \mathcal{P}_{G}^{r-1} $ a set consisting of some nodes covered by $ X $. 
Then there exists an MAT-labeling $ \lambda_{T} $ of $ G[\cup_{Y \in T}Y] $ satisfying $ \lambda_{T}|_{E_{G[Y]}} = \lambda_{Y} $ for any $ Y \in T $.  
\end{claim}
\begin{proof}[Proof of Claim \ref{cl:construction}.]
We prove by induction on $ |T| $.
If $ |T| = 1 $, then it is clear. 
Suppose $ |T| \geq 2 $. 
By Lemma \ref{leaf}, there exist distinct $ X_{0}, Y_{0} \in T $ such that $ X_{0} \cap Y_{0} \supseteq X_{0} \cap Y $ for all $ Y \in T \setminus \{X_{0}\} $. 
By the induction hypothesis on $ |T| $, there exists an MAT-labeling $ \lambda^{\prime} $ of $ G[\cup_{Y \in T\setminus\{X_{0}\}}Y] $ such that $ \lambda^{\prime}|_{E_{G[Y]}} = \lambda_{Y} $ for any $ Y \in T\setminus\{X_{0}\} $. 
Note that
\begin{align*}
X_{0} \cap \left(\bigcup_{Y \in T \setminus\{X_{0}\}}Y\right)
= \bigcup_{Y \in T \setminus\{X_{0}\}}(X_{0} \cap Y) = X_{0} \cap Y_{0}.
\end{align*}
By Lemma \ref{merge}, there exists an MAT-labeling $ \lambda_{T} $ of $ G[\cup_{Y \in T}Y] $ such that $ \lambda_{T}|_{E_{G[X_{0}]}} = \lambda_{X_{0}} $ and $ \lambda_{T}|_{E_{G[\cup_{Y \in T\setminus\{X_{0}\}}Y]}} = \lambda^{\prime} $. 
Therefore $ \lambda_{T}|_{E_{G[Y]}} = \lambda_{Y} $ for any $ Y \in T $. 
\end{proof}
Now we return to the proof of Lemma \ref{node MAT-labeling}. 
Let $ T $ be the set consisting of all nodes covered by $ X $. 
Then use Lemma \ref{MAT-L-PEO of complete graph}(\ref{MAT-L-PEO of complete graph 2}) to extend $ \lambda_{T} $ to $ \lambda_{X} $ of $ G[X] $. 
\end{proof}

We are ready to prove the main result of this subsection.
\begin{theorem}[Strong chordality implies MAT-freeness]
\label{thm:main-2} 
If $ G $ is a strongly chordal graph, then $G$ admits an MAT-labeling. 
\end{theorem}
\begin{proof}
By Lemma \ref{node MAT-labeling}, there exists a family $ \{\lambda_{X}\}_{X \in \mathcal{P}_{G}} $ consisting of MAT-labelings $ \lambda_{X} $ of $ G[X] $ such that $ \lambda_{X}|_{E_{G[Y]}} = \lambda_{Y} $ whenever $ X \supseteq Y $. 
Considering the antichain $\mcK(G)$ of $ \mathcal{P}_{G} $ consisting of the maximal cliques of $ G $, we can construct an MAT-labeling $ \lambda $ of $ G $ by using Lemma \ref{leaf} and the ``gluing trick" (Theorem \ref{glue}). 
More precisely, we show that for any $T \subseteq \mcK(G)$, there exists an MAT-labeling $ \lambda_{T}:= \cup_{Y \in T} \lambda_Y$ of $ G[\cup_{Y \in T}Y] $ satisfying $ \lambda_{T}|_{E_{G[Y]}} = \lambda_{Y} $ for every $ Y \in T $. 
This can be done by induction on $ |T| $ very similar to the proof of Claim \ref{cl:construction}. 
Then take $T = \mcK(G)$.
\end{proof}

%******************************************************************************** 
Finally we present the proofs of the main result of the paper and its corollary. 

\begin{proof}[Proof of Theorem \ref{thm:MAT-free-strong-chordal}.]
It follows from Theorems \ref{thm:main-1}, \ref{thm:main-2} and Proposition \ref{prop:MATF=MATL}.
\end{proof}

\begin{proof}[Proof of Corollary \ref{cor:close-localization}.]
Taking localization on a flat of a graphic arrangement is equivalent to taking an induced subgraph of the underlying graph. 
The proof follows from Theorem \ref{thm:MAT-free-strong-chordal} and a simple fact that the class of strongly chordal graphs is closed under taking induced subgraphs. 
\end{proof}

We close this section by giving an example to illustrate the construction in Theorem \ref{thm:main-2}. 
\begin{example}
\label{ex:7}  
Let $G$ be a unit interval graph in Figure \ref{fig:7-ver}. Its clique intersection poset $\mcP_G$ is given in  Figure \ref{fig:PG-7-ver}. 
First we need to find a family $\F(\mathcal{P}_{G})= \{\lambda_{X}\}_{X \in \mathcal{P}_{G}} $ consisting of MAT-labelings one for each $ G[X] $ such that $\F(\mathcal{P}_{G})$ is closed under inclusion mentioned in Lemma \ref{node MAT-labeling}. 
This can be done inductively from the bottom to top starting from the minimum element $\hat{0}$. 
For example, to find a desired MAT-labeling $\lambda_3 \in \F(\mathcal{P}_{G})$ of $G[X]$ where $X=\{v_2,v_3,v_4, v_5\}$ provided that the compatible MAT-labelings of $G[Y]$ for all $Y$'s covered by $X$ (in this case $\{v_4, v_5\}$ and $\{v_2,v_3,v_4\}$) were given, we use Lemma \ref{merge} (and Lemma \ref{MAT-L-PEO of complete graph}(\ref{MAT-L-PEO of complete graph 2}) if $\cup Y \subsetneq X$). 
Combining the resulting MAT-labelings $\lambda_i\in \F(\mathcal{P}_{G})$ ($1 \le i \le 4$) of the maximal cliques by the   ``gluing trick" (Theorem \ref{glue}) yields an MAT-labeling of $G$. 
Figure \ref{fig:exp-grow-SC} shows a gluing $((\lambda_1\cup \lambda_2) \cup \lambda_3)\cup \lambda_4$ and how the exponents change in each inductive step, which we call an ``exponent growth process". 
Note that although MAT-labeling of $G$ is uniquely determined by $\lambda_i$'s, gluing order is not necessarily unique. 
For example, the gluing $\lambda_1\cup (\lambda_2 \cup (\lambda_3\cup \lambda_4))$ derived from the same method gives the same output but different exponent growth process: $\{0,1,2,3\} \to \{0,1,2,3,3\} \to \{0,1,2,2,3,3\} \to \{0,1,2,2,2,3,3\}$.
\end{example}

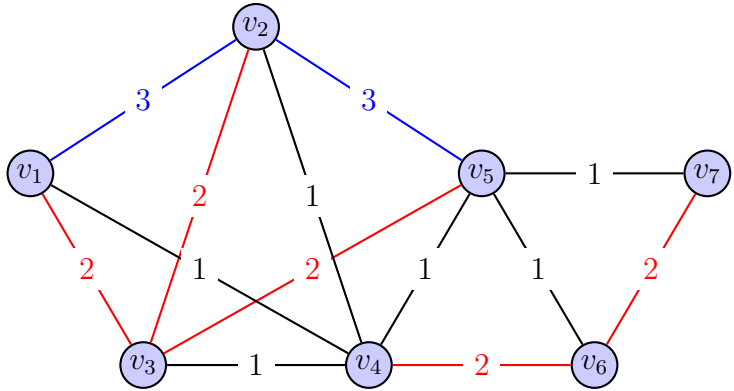
\begin{figure}[htbp]
\centering
\begin{tikzpicture}[scale=1.5]]
\begin{scope}[every node/.style={circle,thick,draw, inner sep=1.7pt, fill=blue!20}]
    \node (v3) at (0,0) {$v_3$};
    \node (v4) at (2,0) {$v_4$};
    \node (v6) at (4,0) {$v_6$};
    \node (v1) at (-1,1.7) {$v_1$};
    \node (v5) at (3,1.7) {$v_5$};
    \node (v7) at (5,1.7) {$v_7$};
        \node (v2) at (1,3) {$v_2$};
\end{scope}
    \Edge[color=red,label=$\color{red}{2}$,style={pos=0.5}](v3)(v1)
    \Edge[color=red,label=$\color{red}{2}$,style={pos=0.5}](v3)(v5)
    \Edge[color=red,label=$\color{red}{2}$,style={pos=0.5}](v3)(v2)
    \Edge[color=red,label=$\color{red}{2}$,style={pos=0.5}](v4)(v6)
    \Edge[color=red,label=$\color{red}{2}$,style={pos=0.5}](v7)(v6)
        \Edge[color=blue,label=$\color{blue}{3}$,style={pos=0.5}](v2)(v1)
    \Edge[color=blue,label=$\color{blue}{3}$,style={pos=0.5}](v2)(v5)
        \Edge[label=$1$,style={pos=0.5}](v4)(v5)
                \Edge[label=$1$,style={pos=0.5}](v6)(v5)
                        \Edge[label=$1$,style={pos=0.5}](v7)(v5)
                                \Edge[label=$1$,style={pos=0.5}](v4)(v1)
                                        \Edge[label=$1$,style={pos=0.5}](v4)(v2)
                                                \Edge[label=$1$,style={pos=0.5}](v4)(v3)
\end{tikzpicture}
\caption{A unit interval (hence strongly chordal) graph $G$ on $7$ vertices with an MAT-labeling constructed by using Theorem \ref{thm:main-2}. The corresponding graphic arrangement $\A_G$ is free with exponents $\{0,1,2,2,2,3,3\}$.}
\label{fig:7-ver}
\end{figure}

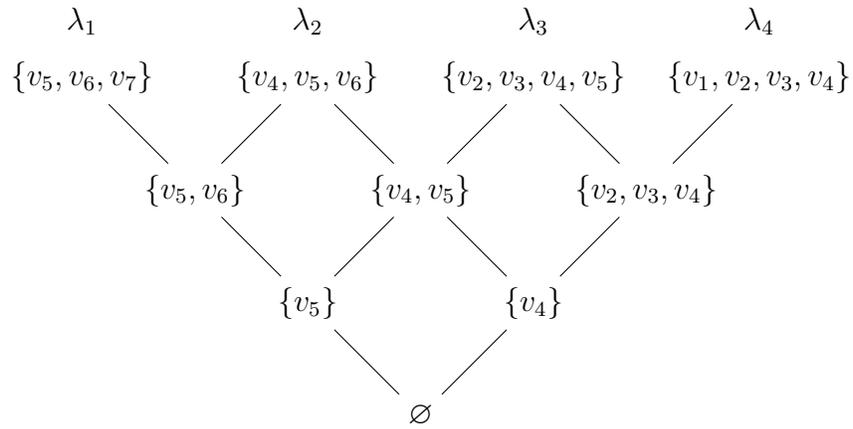
\begin{figure}[htbp]
\centering
\begin{tikzpicture}[scale=1.5]
  \node (1234) at (3,3) {$\{v_1, v_2,v_3,v_4\}$};
          \node (lambda4) at (3,3.5) {$\lambda_4$};
  \node (2345) at (1,3) {$\{v_2,v_3,v_4, v_5\}$};
        \node (lambda3) at (1,3.5) {$\lambda_3$};
  \node (456) at (-1,3) {$\{v_4, v_5,v_6\}$};
      \node (lambda2) at (-1,3.5) {$\lambda_2$};
  \node (567) at (-3,3) {$\{v_5,v_6,v_7\}$};
    \node (lambda1) at (-3,3.5) {$\lambda_1$};
  \node (234) at (2,2) {$\{v_2,v_3,v_4 \}$};
  \node (45) at (0,2) {$\{v_4, v_5\}$};
  \node (56) at (-2,2) {$\{v_5,v_6\}$};
  \node (4) at (1,1) {$\{v_4\}$};
  \node (5) at (-1,1) {$\{v_5\}$};
  \node (0) at (0,0) {$\varnothing$};
  \draw (567) -- (56) -- (5) -- (0) -- (4) -- (234)-- (1234);
    \draw (56) --(456) -- (45) --(4);
        \draw (5) --(45) -- (2345) --(234);
\end{tikzpicture}
\caption{The clique intersection poset of the graph in Figure \ref{fig:7-ver} with MAT-labelings $\lambda_i$ ($1 \le i \le 4$) of the maximal cliques constructed by using Lemma \ref{node MAT-labeling}.}
\label{fig:PG-7-ver}
\end{figure}

\ytableausetup{centertableaux}

\begin{figure}[htbp]
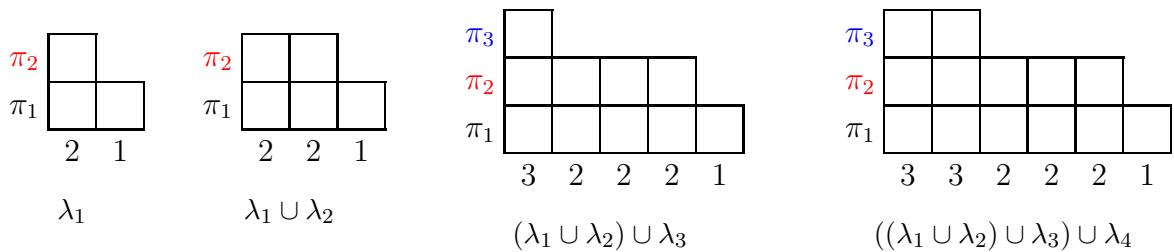

\centering
\begin{subfigure}{.15\textwidth}
  \centering
\begin{ytableau}
  \none[\color{red}{\pi_2}] &  & \none \\
   \none[\pi_1] &  &  \\
  \none & \none[2] & \none[1]
\end{ytableau}
  \caption*{$\lambda_1$}
  \label{fig:1}
\end{subfigure}%
\begin{subfigure}{.2\textwidth}
  \centering
\begin{ytableau}
  \none[\color{red}{\pi_2}] & & & \none \\
   \none[\pi_1] & & &  \\
  \none & \none[2] & \none[2] & \none[1]
\end{ytableau}
  \caption*{$\lambda_1\cup \lambda_2$}  
  \label{fig:2}
\end{subfigure}%
\begin{subfigure}{.30\textwidth}
  \centering
\begin{ytableau}
  \none[\color{blue}{\pi_3}] &  & \none & \none \\
  \none[\color{red}{\pi_2}] & & &  & \\
   \none[\pi_1] & & &  & &\\
  \none & \none[3] & \none[2] & \none[2] & \none[2] & \none[1]
\end{ytableau}
  \caption*{$(\lambda_1\cup \lambda_2) \cup \lambda_3$}  
  \label{fig:3}
\end{subfigure}%
\begin{subfigure}{.35\textwidth}
  \centering
\begin{ytableau}
  \none[\color{blue}{\pi_3}] &  &  & \none \\
  \none[\color{red}{\pi_2}] & & & &  & \\
   \none[\pi_1] & & &  & & &\\
  \none & \none[3] & \none[3] & \none[2] &  \none[2] & \none[2] & \none[1]
\end{ytableau}
  \caption*{$((\lambda_1\cup \lambda_2) \cup \lambda_3)\cup \lambda_4$}  
  \label{fig:4}
\end{subfigure}
\caption{An exponent growth process for  the graph in Figure \ref{fig:7-ver} following the ``gluing trick"  in Theorem \ref{thm:main-2}.}
\label{fig:exp-grow-SC}
\end{figure}

%******************************************************************************** 
\section{Further remarks and open problems}
\label{sec:rem}
In this section we address some remarks and suggest problems for future research.
\begin{enumerate}[(A)]
\item
As noted in Introduction, our Theorem \ref{thm:MAT-free-strong-chordal} gives an alternative proof that the ideal graphic arrangements are MAT-free (type $A$ of Theorem \ref{thm:ideal-free}). 
We give here two examples to illustrate the difference between two methods. 
The original proof of the ideal MAT-free theorem is inductive on the height of ideals  \cite[\S5]{ABCHT16}, and in each inductive MAT-step only some of maximal exponents get increased by $1$. 
This yields a rigorous exponent growth process hence differs from  our construction in Theorem \ref{thm:main-2}. 
For example, the unit interval graph $G$ in Figure \ref{fig:7-ver} with the given vertex-labeling has its corresponding graphic arrangement $\A_G$ an ideal subarrangement of the Weyl arrangement $\A_{\Phi^+(A_{6})}$. 
The exponent growth process following the ideal MAT-free theorem is given in Figure \ref{fig:exp-grow-IF} which differs from that in Figure \ref{fig:exp-grow-SC}. 
Our construction applies also to strongly chordal graphs that are not unit interval graphs. 
Another way to see the difference between two methods is to consider MAT-labelings of complete graphs, see Remark \ref{rem:complete}, Lemma \ref{MAT-L-PEO of complete graph}(\ref{MAT-L-PEO of complete graph 2}) and Figure \ref{fig:MAT-K4}. 

\begin{figure}[htbp]
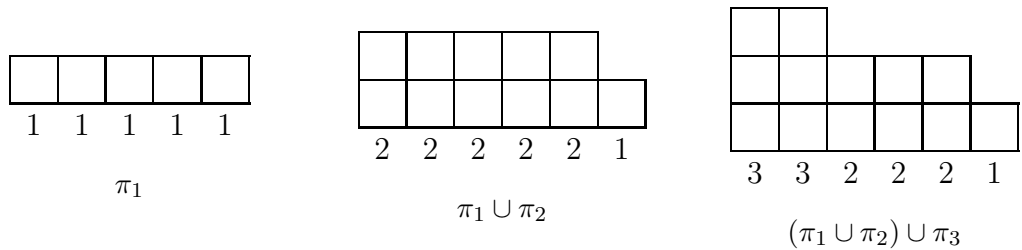

\centering
\begin{subfigure}{.30\textwidth}
  \centering
\begin{ytableau}
~ & &  & &\\
 \none[1] & \none[1] & \none[1] & \none[1] & \none[1]
\end{ytableau}
  \caption*{$\pi_1$}  
%  \label{fig:i-1}
\end{subfigure}%
\begin{subfigure}{.30\textwidth}
  \centering
\begin{ytableau}
~& & &   & \\
~& & &  &  &\\
 \none[2] & \none[2] & \none[2] &  \none[2] & \none[2] & \none[1]
\end{ytableau}
  \caption*{$\pi_1\cup \pi_2$}  
%  \label{fig:i-2}
\end{subfigure}%
\begin{subfigure}{.30\textwidth}
  \centering
\begin{ytableau}
 ~& &   \none \\
~& & &   & \\
~& & &  &  &\\
 \none[3] & \none[3] & \none[2] &  \none[2] & \none[2] & \none[1]
\end{ytableau}
  \caption*{$(\pi_1\cup \pi_2) \cup \pi_3$}  
%  \label{fig:i-3}
\end{subfigure}
\caption{Exponent growth process for  the graph in Figure \ref{fig:7-ver} following the ideal MAT-free theorem.}
\label{fig:exp-grow-IF}
\end{figure}

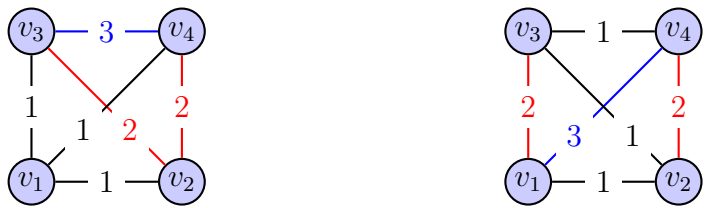
\begin{figure}[htbp]
\centering
\begin{subfigure}{.4\textwidth}
  \centering
\begin{tikzpicture}[scale=1]]
\begin{scope}[every node/.style={circle,thick,draw, inner sep=1.7pt, fill=blue!20}]
    \node (v1) at (0,0) {$v_1$};
    \node (v2) at (2,0) {$v_2$};
    \node (v4) at (2,2) {$v_4$};
    \node (v3) at (0,2) {$v_3$};
\end{scope}
    \Edge[color=red,label=$\color{red}{2}$,style={pos=0.7}](v3)(v2)
    \Edge[color=red,label=$\color{red}{2}$,style={pos=0.5}](v4)(v2)
        \Edge[color=blue,label=$\color{blue}{3}$,style={pos=0.5}](v3)(v4)
\Edge[label=$1$,style={pos=0.7}](v4)(v1)
\Edge[label=$1$,style={pos=0.5}](v1)(v2)
\Edge[label=$1$,style={pos=0.5}](v1)(v3)
\end{tikzpicture}
%  \caption*{}  
\end{subfigure}%
\begin{subfigure}{.4\textwidth}
  \centering
 \begin{tikzpicture}[scale=1]]
\begin{scope}[every node/.style={circle,thick,draw, inner sep=1.7pt, fill=blue!20}]
    \node (v1) at (0,0) {$v_1$};
    \node (v2) at (2,0) {$v_2$};
    \node (v4) at (2,2) {$v_4$};
    \node (v3) at (0,2) {$v_3$};
\end{scope}
    \Edge[color=red,label=$\color{red}{2}$,style={pos=0.5}](v3)(v1)
    \Edge[color=red,label=$\color{red}{2}$,style={pos=0.5}](v4)(v2)
        \Edge[color=blue,label=$\color{blue}{3}$,style={pos=0.25}](v1)(v4)
\Edge[label=$1$,style={pos=0.5}](v4)(v3)
\Edge[label=$1$,style={pos=0.5}](v1)(v2)
\Edge[label=$1$,style={pos=0.25}](v2)(v3)
\end{tikzpicture}
\end{subfigure}
\caption{Two nonisomorphic MAT-labelings of $K_4$ constructed by using Lemma \ref{MAT-L-PEO of complete graph}(\ref{MAT-L-PEO of complete graph 2}) (left) and the ideal MAT-free theorem (right).}
\label{fig:MAT-K4}
\end{figure}

\item Cuntz-M{\"u}cksch  \cite[Example 22]{CM20} showed that MAT-freeness is in general not closed under taking restriction. Their example is a non-MAT-free restriction  to a hyperplane of the Weyl arrangement of type $E_6$. 
We give here a different example (with a smaller number of hyperplanes) thanks to the fact that the class of strongly chordal graph is not closed under taking edge-contraction. 
Consider the rising sun (which is a strongly chordal graph) with its edge $e$ displayed in Figure \ref{fig:rsun}. 
 Taking the contraction of $e$ results in the $3$-sun which is not strongly chordal. 
 
 \begin{figure}[htbp]
\centering
\begin{tikzpicture}
\draw (0,0) node[v](x1){};
\draw (1,0) node[v](x2){};
\draw (1,1) node[v](x3){};
\draw (0,1) node[v](x4){};
\draw (1.865,0.5) node[v](y2){};
\draw (0.5,1.865) node[v](y3){};
\draw (-0.865,0.5) node[v](y4){};
\draw (x2)--(x3)--(x4)--(x1)--cycle;
\draw (x1)--(x3);
\draw (x2)--(x4);
\draw (x2)--(y2)--(x3)--(y3)--(x4)--(y4)--(x1);
\draw (x1)--(x2) node [midway, below] {$e$};
%\draw (0.5,-0.5) node(){rising sun};
\end{tikzpicture}\caption{The rising sun.}
\label{fig:rsun}
\end{figure}
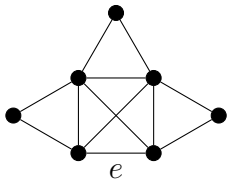

\item Strongly chordal graphs are the intersection graphs of unit balls in $ \mathbb{R} $-trees \cite{kuroda2021unit-gac}. 
Therefore they can be considered as generalization of unit interval graphs in the perspective of intersection graphs.

\item 
Strongly chordal graphs are also known as the graphs having a \textbf{strong perfect elimination ordering (SPEO)} \cite{Far83}, i.e., a PEO $ (v_{1}, \dots, v_{\ell}) $ with the property that for all $i<j, k<q$ if $\{v_i, v_k\}, \{v_i, v_q\} , \{v_j, v_k\}$ are edges, then $ \{v_j, v_q\}$ is an edge.
It would be interesting to find a (more direct) connection between SPEO and MAT-PEO.

\item If an arrangement $\A$ is MAT-free, then $\A$ is \emph{accurate}  \cite[Theorem 1.2]{MR21} i.e., $\A$ is free  with $\exp(\A) =\{d_1, \ldots, d_\ell\}_{\le}$ and there exists for each $0 \le p <\ell$ a $p$-codimensional flat $X\in L(\A)$ such that $\A^X$ is free with  $\exp(\A^X)=\{d_1,\ldots, d_{\ell-p}\}_{\le}$.
Characterize the accuracy of graphic arrangements. 
We are able to show that if $G$ is an $n$-sun, then $\A_G$ is accurate (but not MAT-free).
\item From  Theorems \ref{thm:MAT-free-strong-chordal} and \ref{Nevries-Rosenke}, we now know that the MAT-freeness of graphic arrangements can be characterized by a poset structure, the clique intersection poset of chordal graphs. 
Define a ``clique intersection poset" of an arbitrary (supersolvable) arrangement and characterize the MAT-freeness of the arrangement by the poset.
It is related to another question of Cuntz-M{\"u}cksch  \cite[Problem 47]{CM20} which asked if the MAT-freeness can be characterized by a partial order on the hyperplanes, generalizing the classical partial order (\S\ref{subsec:MATfree-arr}) on the positive roots of an irreducible root system. 
\end{enumerate}

%******************************************************************************** 
\vskip1em
\noindent
\textbf{Acknowledgements.} 
The first author was supported by JSPS Research Fellowship for Young Scientists Grant Number 19J12024, and is currently supported by a postdoctoral fellowship of the Alexander von Humboldt Foundation. 
We thank Professor Takuro Abe for posing a question that whether or not strongly chordal graphs are closed under taking contraction which guides us to the rising sun example in \S\ref{sec:rem}(B).
We thank Professor Gerhard R\"{o}hrle for suggesting the problem of characterizing accurate graphic arrangements in \S\ref{sec:rem}(E).

\bibliographystyle{amsplain}
\bibliography{references}

\end{document}